\documentclass[11pt,reqno]{amsart} %
\usepackage{amsmath,amscd}
\usepackage{amsthm}
\usepackage{amssymb}
\usepackage{amsfonts}
\usepackage{latexsym}
\usepackage{url}
\usepackage{latexsym}
\usepackage{graphicx,color,subfig}
\usepackage{enumerate,fancybox}%
\usepackage[colorlinks=false, citecolor=blue]{hyperref} % to make \ref and \eqref appear as a colored link
\newtheorem{theorem}{Theorem}[section]

\newtheorem{lemma}{Lemma}[section]

\newtheorem{corollary}{Corollary}[section]

\theoremstyle{remark}
\newtheorem{remark}{Remark}[section]

\numberwithin{equation}{section}

\def\ole{\preceq}

\begin{document}
\title[Error analysis of BKM/AB]{Error analysis of the Bergman kernel method with singular basis functions}
\date{\today}

\author[M.\ Lytrides]{M.\ A.\ Lytrides}
\email{map6lm1@ucy.ac.cy}
%\urladdr{}
\address{Department of Mathematics and Statistics\\
         University of Cyprus\\
         P.O. Box 20537\\
         1678 Nicosia\\
         Cyprus}
\author[N.\ Stylianopoulos]{\ N.\ S.\ Stylianopoulos}
\email{nikos@ucy.ac.cy}
\urladdr{http://www.ucy.ac.cy/nikos.html}
\address{Department of Mathematics and Statistics\\
         University of Cyprus\\
         P.O. Box 20537\\
         1678 Nicosia\\
         Cyprus}



% fill out if necessary or keep empty, acknowledgements go before bibliography
%\thanks{Research supported by XYZ}
\begin{abstract}
Let $G$ be a bounded Jordan domain in the complex plane with
piecewise analytic boundary. We present theoretical
estimates and numerical evidence for certain phenomena,
regarding the application of the Bergman kernel method with algebraic
and pole singular basis functions, for approximating the conformal mapping of
$G$ onto the normalized disk. In this way, we complete the task of providing full
theoretical justification of this method.
\end{abstract}
\maketitle

\section{Introduction}\label{sec:1}
Let $G$ be a bounded, simply-connected domain in the complex plane
$\mathbb{C}$ whose boundary $\Gamma:=\partial G$ is a Jordan curve and let
$\Omega:=\overline{\mathbb{C}}\setminus\overline{G}$ denote the complement of $\overline{G}$
with respect to the extended complex plane. Fix
$z_{0}\in G$ and let $f_{0}$ denote the conformal map of
$G$ onto the disk $D(0,r_{0}):=\{z:|z|<r_{0}\}$, normalized by the conditions
$f_{0}(z_{0})=0$ and $f'_{0}(z_{0})=1$. The quantity $r_{0}:=r_{0}(G,z_{0})$
is called the conformal radius of $G$ with respect to $z_{0}$.

For the inner product
\begin{equation}\label{eq:1}
\langle f,g\rangle:=\int_{G}f(z)\,\overline{g(z)}\,dA(z),
\end{equation}
where $dA$ denotes the differential of the area measure on $\mathbb{C}$, we consider the
Hilbert space
\begin{equation}\label{eq:2}
L^{2}_{a}(G):=\big\{f:f\,\, \textmd{analytic}\,\, \textmd{in}\,\,
G,\,\,\langle f,f\rangle<\infty\big\},
\end{equation}
with corresponding norm $\|f\|_{L^{2}(G)}:=\langle f,f\rangle^{\frac{1}{2}}$.

Let $K(\cdot,z_{0})$ denote the \textit{Bergman kernel function} of $G$ with
respect to $z_{0}$. This is the unique function of $L^{2}_{a}(G)$ satisfying the reproducing property
\begin{equation}\label{eq:3}
\langle g,K(\cdot,z_{0})\rangle=g(z_{0}),\,\, \textmd{for}\,\,\textmd{all}\,\, g\in L^{2}_{a}(G).
\end{equation}
It follows from (\ref{eq:3}) that the kernel $K(\cdot,z_{0})$ is related to the mapping function $f_{0}$
by means of
\begin{equation}\label{eq:4}
f^\prime_{0}(z)=\frac{K(z,z_{0})}{K(z_{0},z_{0})}\quad \textmd{and}\quad
f_{0}(z)=\frac{1}{K(z_{0},z_{0})}\int_{z_{0}}^{z}K(\zeta,z_{0})d\zeta,
\end{equation}
see e.g.\ \cite[p.\,33]{Gabook87}.
These yield the two relations,
\begin{equation}\label{eq:5}
K(z,z_{0})=\frac{1}{\pi r_{0}^{2}}f^\prime_{0}(z)\quad\textup{and}\quad
r_{0}=\frac{1}{\sqrt{\pi K(z_0,z_{0})}}.
\end{equation}

Now, let $\{P_n\}_{n=0}^{\infty}$ denote the sequence of the \textit{Bergman polynomials} of
$G$. This is defined as the sequence of polynomials
\begin{equation}
P_n(z) = \lambda_n z^n+ \cdots, \quad \lambda_n>0,\quad n=0,1,2,\ldots,
\end{equation}
that are orthonormal with respect to the inner product (\ref{eq:1}), i.e.,
\begin{equation}
\int_{G}P_{m}(z)\overline{P_{n}(z)}dA(z)=\delta_{m,n}.
\end{equation}
The Bergman polynomials form a complete orthonormal system in $L^{2}_{a}(G)$.
Therefore, in view of the reproducing property (\ref{eq:3}),
\begin{equation}\label{eq:Kseries}
K(z,z_{0})=\sum_{j=0}^{\infty}\overline{P_{j}(z_{0})} P_{j}(z),
\end{equation}
locally uniformly with respect to $z\in G$.
%The standard way to construct them is  by applying the Gram-Schmidt process to the monomials $\{z^{n}\}_{n=0}^{\infty}$.

The \emph{Bergman kernel method} (BKM) is an orthonormalization method for computing approximations to the conformal map $f_{0}(z)$.
It is based on the fact that the kernel $K(z,z_0)$ is given explicitly in terms of the Bergman polynomials $\{P_n(z)\}_{n=0}^{\infty}$. Thus, the partial sums of the Fourier series expansion of $K(z,z_{0})$
are given by
\begin{equation}\label{eq:kernel3}
K_{n}(z,z_{0})
%:=\sum_{j=0}^{n}\langle K(\cdot,z_{0}),P_{j}\rangle P_{j}(z):
:=\sum_{j=0}^{n}\overline{P_{j}(z_{0})} P_{j}(z),\quad n\in\mathbb{N}.
\end{equation}
The polynomials $\{K_{n}(z,z_0)\}_{n=0}^{\infty}$ are the so-called \textit{kernel polynomials} of $G$, with
respect to $z_0$. They provide the best $L^2(G)$- approximation to $K(\cdot,z_0)$ out of the space $\mathbb{P}_{n}$ of complex polynomials of degree at most $n$.

In accordance with (\ref{eq:4}), the $n$-th BKM approximation to $f_{0}$ is given by
\begin{equation}\label{eq:bib1}
\pi_{n}(z):=\frac{1}{K_{n-1}(z_{0},z_{0})}\int_{z_{0}}^{z}K_{n-1}(\zeta,z_{0})d\zeta,\quad n\in\mathbb{N}.
\end{equation}
This defines the sequence $\{\pi_{n}\}_{n=1}^{\infty}$ of the  \textit{Bieberbach polynomials} of $G$, with respect to $z_0$. The polynomial $\pi_n$ solves the following minimal problem:
Let
\begin{equation*}
\mathbb{P}_{n}^{*}:=\{p:p\in\mathbb{P}_{n},\,\, \textup{with}\,\,
p(z_{0})=0\,\, \textup{and}\,\, p^\prime(z_{0})=1\}.
\end{equation*}
Then, for each $n\in\mathbb{N}$, \textit{the polynomial $\pi_{n}$ minimizes
uniquely the two norms} $\|f_{0}^{\prime}-p'\|_{L^{2}(G)}$ \textit{and} $\|p^\prime\|_{L^{2}(G)}$
\textit{over all} $p\in\mathbb{P}_{n}^{*}$; see e.g.\ \cite[Kap.~III, \S1]{Ga64}.
%see e.g. \cite[p.\,34]{Ga}.

Regarding the convergence of the method, we note that in cases when $f_{0}$ has an analytic continuation across $\Gamma$,
then this is a consequence of Walsh's theory of maximal convergence \cite[\S4.7, \S5.3]{Wa}. In order to be more specific,
let $\Phi$ denote the conformal map of $\Omega$ onto $\Delta:=\{w: |w|>1\}$, normalized so that near infinity,
\begin{equation}\label{eq:Phiatinf}
\Phi(z)=\gamma z +\gamma_{0}+\frac{\gamma_{1}}{z}+\frac{\gamma_{2}}{z^{2}}+\cdots,\quad\gamma>0.
\end{equation}
Note that $\gamma=1/\textup{cap}(\Gamma)$, where $\textup{cap}(\Gamma)$ denotes the (logarithmic) capacity of
$\Gamma$.
%\begin{equation}\label{error2}
%E_{n,2}(K,G):=\|K(\cdot,z_{0})-K_{n}(\cdot,z_{0})\|_{L^{2}(G)},
%\end{equation}
Then,
\begin{equation}\label{error1}
%E_{n,\infty}(f_{0},G):=
\|f_{0}-\pi_{n}\|_{L^{\infty}(\overline{G})}=O\left(\frac{1}{R^n}\right),
\end{equation}
holds for any $1<R<|\Phi(z_1)|$, but for no $R>|\Phi(z_1)|$, where $z_1$ denotes the nearest singularity of $f_{0}$ in $\Omega$; see also \cite[Ch.\ I]{Gabook87}. (We use $\|\cdot\|_{L^{\infty}(\overline{G})}$ to denote the sup-norm on
$\overline{G}$.)

In cases when $\Gamma$ is piecewise analytic and $f_{0}$ has singularities on $\Gamma$, then Levin, Papamichael and Siderides were the first to observe in \cite{LPS} that the error (\ref{error1}) depends on the boundary singularities of the mapping function $f_{0}$ on $\Gamma$, and also on the singularities of the extension of $f_{0}$ across the segments of $\Gamma$ into $\Omega$. Accordingly, in order to improve the numerical performance of the BKM, they extended the method by orthonormalizing a system of basis functions consisting from monomials, as in the BKM, and also from functions that reflect the dominant singularities of $f_{0}$ on $\Gamma$ and in $\Omega$.  This  extension is known as BKM/AB (AB stands for \textit{augmented basis}). The BKM/AB was used subsequently in \cite{PK81} and \cite{PW86}.

The most precise results regarding the convergence of the BKM are due to D.\ Gaier \cite{Ga98}.
In particular, under the assumption
that $\Gamma$ is piecewise analytic without cusps, Gaier derived the estimate
\begin{equation}\label{eq:6}
\|f_{0}-\pi_{n}\|_{L^{\infty}(\overline{G})}=O(\log n)\frac{1}{n^{s}},
\end{equation}
where $s:=\lambda/(2-\lambda)$ and $\lambda\pi$ $(0<\lambda<2)$ denotes the smallest exterior angle
where two analytic arcs of $\Gamma$ meet. Regarding sharpness of the estimate (\ref{eq:6}), it was shown in
\cite[Thm.\ 4]{Ga88} that there are cases where the exponent $s$ can not be replaced by a smaller number.
However, the factor $\log n$ can be replaced by $\sqrt{\log n}$, see \cite{AG} and \cite[Rem.\ 3.1]{MSS}.
A lower estimate of the form
\begin{equation}\label{eq:MSS3.19}
\|f_{0}-\pi_{n}\|_{L^{\infty}(\overline{G})}\ge c\,\frac{1}{n^{s}},
\end{equation}
provided that $1/(2-\lambda)$ is not a positive integer,
where $c$ is a constant that does not depend on  $n$, was established in \cite[Thm.\ 3.2]{MSS} by Maymeskul, Saff and the second author.

The theoretical justification of the BKM/AB with basis function that reflect the corner singularities of $f_0$
was given in \cite{MSS}, by means of sharp estimates for the associated BKM/AB errors. The purpose of the present paper is to derive theoretical results that justify the use of basis functions that reflect (a) pole singularities of $f_0$ and (b)
both corner and pole singularities of $f_0$. More specifically, we derive upper and lower estimates for the
BKM/AB errors in the case (a), and upper estimates for the BKM/AB errors in the case (b).
In doing so, we complete the task that was put forward by Yu.\ E.\ Khokhlov, reviewer of the introductory paper \cite{LPS} of the BKM/AB in the Mathematical Reviews, who concluded that: \lq\lq\textit{A proof of the convergence of the numerical method given and an investigation of its convergence rate are lacking, so the results obtained are of a heuristic nature}\rq\rq.

The paper is organized as follows: In Section~2 we set up the notation and recall the BKM/AB. Section~3 is devoted to the study of the various BKM and BKM/AB errors, in cases when $f_0$ has an analytic continuation across $\Gamma$, hence only basis functions reflecting poles are used in the BKM/AB. In Section~4, we consider the case when both corner and pole basis functions are included in BKM/AB. Finally, in Section~5, we present numerical computations that illustrate the theory of Sections 3 and 4.
%and discuss how the theoretical results can be used in order to improve the numerical implementation of the method.

\section{The Bergman kernel method with singular basis functions}\label{sec:2}
\subsection{Corner singularities}\label{subsec:corners}
Throughout this section we assume that the boundary $\Gamma$ of $G$ consist of $N$ analytic arcs that meet at corner points $\tau_{k}$,  $k=1,2,\ldots,N$, where they form interior angles $\alpha_{k}\pi$, $0<\alpha_{k}<2$. Then, we have the following asymptotic expansions for $f_{0}$, valid near $\tau_{k}$:
\begin{itemize}
\item[(i)]
If $\alpha_{k}$ is irrational, then
\begin{equation}\label{eq:lemman}
f_{0}(z)=f_{0}(\tau_{k})+\sum_{p,q}B_{p,q}(z-\tau_{k})^{p+q/\alpha_{k}},
\end{equation}
where $p$ and $q$ run over all integers $p\geq 0$, $q\geq1$ and $B_{0,1}\neq 0$.
\item[(ii)]
If $\alpha_{k}=a/b$, with $a$ and $b$ relative prime numbers, then
\begin{equation}\label{eq:lehman2}
f_{0}(z)=f_{0}(\tau_{k})+\sum_{p,q,m}B_{p,q,m}(z-\tau_{k})^{p+q/\alpha_{k}}(\log(z-\tau_{k}))^{m},
\end{equation}
where $p,q$ and $m$ run over all integers $p\geq 0$, $1\leq q\leq a$, $1\leq m\leq p/b$ and $B_{0,1,0}\neq 0$.
\item[(iii)]
If $\tau_{k}$ is formed by two straight-line segments, then
\begin{equation}\label{eq:lines}
f_{0}(z)=f_{0}(\tau_{k})+\sum_{l=1}^\infty B_{l}(z-\tau_{k})^{l/\alpha_{k}},
\end{equation}
where $B_1\ne 0$. Furthermore, (\ref{eq:lemman}) holds in the case when $\tau_{k}$ is formed by two
circular arcs, or a straight-line and a circular arc.
\end{itemize}
In the above, (i) and (ii) are due to Lehman \cite{Lehman}, while (iii) emerges easily from the reflection principle; see also \cite[\S2.1]{Ga98} and \cite[pp.~6--7]{P1}.

It follows from (iii) that if $G$ is a half-disk or a rectangle, then $f_0$ has a Taylor series expansion valid around each corner, and thus an analytic continuation across $\Gamma$ into $\Omega$. In this case, the only singularities of $f_0$ are simple poles in $\Omega$. This shows that the study of the BKM/AB, even with only pole basis function is important in the applications.

For simplicity in the exposition, we shall assume throughout this paper that \textit{no logarithmic terms occur} in the asymptotic expansion of $f_{0}$ near the corner $\tau_{k},\,\, k=1,2,\ldots,N$. This, for example, will be the case in the expansions (\ref{eq:lemman}) and (\ref{eq:lines}) above. Nevertheless, our method of study
can be adjusted to cover logarithmic singularities as well.

Let $M$ denote the number of corners of $\Gamma$ for which $\alpha_{k}$ is not of the special form $1/m,\,\, m\in\mathbb{N}$. When we present results for corner singularities we shall assume that $M\ge1$. We index such corners by $\tau_{k}$, $k=1,\ldots,M$. That is, if $N>M$, then the mapping function $f_{0}$ has an analytic continuation in some neighborhood of the corner $\tau_{N}$.

For $k=1,\ldots,M$, we denote by $\{\gamma_{j}^{(k)}\}_{j=1}^{\infty}$ the increasing arrangement of the possible powers $p+q/\alpha_{k}$ of $(z-\tau_{k})$ that appear in the asymptotic expansion of $f_{0}(z)$ near $\tau_{k}$.
In particular, if $\tau_{k}$ is formed by two straight-line segments, then $\gamma_{j}^{(k)}=j/\alpha_{k}$, $j=1,2,\ldots$.
Also, if $\alpha_{k}$ is irrational, or the corner $\tau_{k}$ is formed by
two circular arcs, then
\begin{eqnarray*}
\gamma_1^{(k)}&=&1/\alpha_k;\\
\gamma_2^{(k)}&=&1/\alpha_k+\min(1/\alpha_k,1);\\
\gamma_3^{(k)}&=&
\left\{
\begin{array}{ll}
1/\alpha_k+2,\quad&0<\alpha_k<1/2,\\
2/\alpha_k,       &1/2<\alpha_k<1,\\
1/\alpha_k+1,     &1<\alpha_k<2;
\end{array}
\right.
\\
\vdots&
\end{eqnarray*}
\begin{remark}\label{rem:Lehman}
Under the assumption regarding the no-appearance of logarithmic terms, the asymptotic expansion near
$\tau_{k}$, $k=1,2,\ldots,M$, can be written in the form
\begin{equation}\label{eq:lemman2}
f_{0}(z)=\sum_{j=0}^{\infty}a_{j}^{(k)}(z-\tau_{k})^{\gamma_{j}^{(k)}},
\end{equation}
where, $\gamma_{0}^{(k)}:=0$ and $a_{1}^{(k)}\neq0$. Note that, we always have $\gamma_{1}^{(k)}>1/2$, and since $\tau_k$ is not a special corner, $\gamma_{1}^{(k)}\notin\mathbb{N}$. Therefore $(z-\tau_{k})^{\gamma_{1}^{(k)}}$ has an algebraic singularity at $\tau_k$. However, when $\alpha_{k}$ is rational, it is possible that $\gamma_{j}^{(k)}\in \mathbb{N}$, for indices $j\geq 2$, so that
$(z-\tau_{k})^{\gamma_{j}^{(k)}}$ is analytic at $\tau_{k}$.
\end{remark}

\subsection{Pole singularities}\label{subsec:poles}
Since $f_0(z_0)=0$, $z_0\in G$, it follows from the reflection principle for analytic arcs that the extension of $f_0$ across any segment constituting $\Gamma$ would have a pole or a pole-type singularity at the reflected images of $z_0$. For example, if $\Gamma$ consists explicitly from straight-line segments and/or circular arcs, then $f_{0}$ has a simple pole (due to the univalency of $f_0$) at every mirror image of $z_{0}$ (with respect to the straight-lines) and at every geometric inverse of $z_{0}$ (with respect to the circular arcs), that lies in $\Omega$.
More generally, $f_{0}$ may have at points $z_{j}\in\Omega$ a pole, or a poly-type, singularity of the form
\begin{equation}\label{eq:papam}
(z-z_{j})^{-{k_{j}}/{m_{j}}},\quad k_j,m_j\in\mathbb{N}.
\end{equation}
According to \cite[\S~5.1]{PWH}, the following three special cases occur frequently in the applications:
\begin{itemize}
\item[(i)]
$k_j=m_j=1$. In this case, $f_{0}$ has a simple pole at $z_{j}$.
\item[(ii)]
$k_j=2$, $m_j=1$. In this case, $f_{0}$ has a double pole at $z_{j}$.
\item[(iii)]
$k_j=1$, $m_{j}=2$. In this case, $f_{0}$ has a rational pole singularity at $z_{j}$.
\end{itemize}

In order to describe the BKM/AB, we assume that the nearest singularities of $f_{0}$ in $\Omega$ are poles or rational poles, of the form (\ref{eq:papam}) at points $z_{j}$, $j=1,2,\ldots \kappa$, where
$|\Phi(z_{1})|\le|\Phi(z_{2})|\le\ldots\le|\Phi(z_{\kappa})|$ and that the other singularities of $f_{0}$ in $\Omega$
occur at points $z_{\kappa+1},z_{\kappa+2},\ldots$, where
$|\Phi(z_k)|<|\Phi(z_{\kappa+1})|\le|\Phi(z_{\kappa+2})|\le\ldots$.
\subsection{BKM/AB}
Using the above notation, the BKM/AB with $n$ monomials, $\kappa$ poles and $p_{k}$ corner singularities at each (non-special) corner $\tau_{k}$, $k=1,2\ldots,M$, can be summarized as follows:
\begin{itemize}
\item[(i)]
Start with the augmented system $\{\eta_{j}\}$ consisting of:
\begin{enumerate}
\item
the nearest poles or rational poles, i.e., for $j=1,2\ldots \kappa$,
\begin{equation}\label{eq:7}
\eta_{j}(z)=\bigg[\bigg(\frac{1}{z-z_{j}}\bigg)^{k_{j}/m_{j}}\bigg]';
\end{equation}
\item
the dominant $r_{M}:=\sum_{k=1}^{M}p_{k}$ algebraic singular functions, i.e., for each non-special corner
$\tau_k$, $k=1,2,\ldots,M$,
\begin{equation}\label{eq:8}
\eta_{\kappa+j}(z)=[(z-\tau_{k})^{\gamma_{j}^{(k)}}]',\quad j=1,2,\ldots p_{k};
\end{equation}
\item
the monomials
\begin{equation}\label{eq:9}
\eta_{\kappa+r_{M}+j}(z)=(z^j)^\prime,\quad j=1,2,\ldots,n.
\end{equation}
\end{enumerate}
(As it was noted in Remark~\ref{rem:Lehman}, it might be possible that $\gamma_j^{(k)}\in\mathbb{N}$. If this happens, we avoid
redundancy in the basis by omitting such $\gamma_j^{(k)}$.)
\item [(ii)]
Orthonormalize $\{\eta_{j}\}$, by means of the Gram-Schmidt process
to produce the orthonormal set $\{\widetilde{P}_{j}\}$, where
\begin{equation}\label{eq:10}
\widetilde{P}_{j}(z)=\sum_{i=1}^{j}b_{j,i}\,\eta_{i}(z),\quad b_{j,j}>0.
\end{equation}
\item [(iii)]
Approximate $K(z,z_{0})$ by its finite Fourier expansion with respect to  $\{\widetilde{P}_{j}\}$:
\begin{equation}\label{eq:ker}
\widetilde{K}_{n}(z,z_{0}):=\sum_{j=1}^{\kappa+r_{M}+n}\overline{\widetilde{P}_{j}(z_{0})}
\widetilde{P}_{j}(z)=\sum_{j=1}^{\kappa+r_{M}+n}d_{n,j}\eta_{j}(z).
\end{equation}
\item [(iv)]
Approximate  $f_{0}(z)$ by
\begin{equation}\label{eq:11}
\widetilde{\pi}_{n+1}(z):=\frac{1}{\widetilde{K}_{n}(z_{0},z_{0})}
\int_{z_{0}}^{z}\widetilde{K}_{n}(\zeta,z_{0})d\zeta=\sum_{j=1}^{\kappa+r_{M}+n}c_{n,j}\mu_{j}(z),
\end{equation}
where
\begin{equation}\label{eq:mudef}
\mu_j(z):=
\int_{z_{0}}^{z}\eta_j(\zeta)d\zeta.
\end{equation}
\end{itemize}

We call the functions $\{\widetilde{P}_{j}\}$ the \emph{augmented Bergman polynomials} of $G$,
with respect to $\{\eta_{j}\}$, and the functions $\{\widetilde{\pi}_n\}$ the \textit{augmented Bieberbach polynomials} over the system $\{\mu_j\}$.
Clearly, $\widetilde{\pi}_{n}(z_{0})=0$ and $\widetilde{\pi}^{\prime}_n(z_{0})=1$, $n\in\mathbb{N}$.
Note that $\{\widetilde{P}_{j}\}_{j=1}^{\infty}$ forms a complete orthonormal system in $L_a^2(G)$. Consequently,
\begin{equation}\label{eq:Kseriestil}
K(z,z_{0})=\sum_{j=1}^\infty\overline{\widetilde{P}_{j}(z_{0})}
\widetilde{P}_{j}(z),
\end{equation}
locally uniformly with respect to $z\in G$, cf.\ (\ref{eq:Kseries}).

We conclude this section by presenting a result which shows that the two errors
$\|f_{0}^{\prime}-\widetilde{\pi}_{n+1}^{\prime}\|_{L^{2}(G)}$ and
$\|K(\cdot,z_{0})-\widetilde{K}_{n}(\cdot,z_{0})\|_{L^{2}(G)}$ are of the same order.
This fact will be used below in Sections $3$ and $4$.

In what follows we denote by  $c$, $c_1$, $c_2,\dots$, constants that are independent of $n$.
For quantities $A>0$, $B>0$, we use the notation $A\ole B$ (inequality with respect to the order) if
$A\leq cB$. The expression $A\asymp B$ means that $A\ole B$ and $B\ole A$
simultaneously.

\begin{lemma}\label{l:1}
It holds that,
\begin{equation}\label{eq:st1}
\|f_{0}^{\prime}-\widetilde{\pi}_{n+1}^{\prime}\|_{L^{2}(G)}\asymp
\|K(\cdot,z_{0})-\widetilde{K}_{n}(\cdot,z_{0})\|_{L^{2}(G)}.
\end{equation}
\end{lemma}
\begin{proof} We set $m:=\kappa+r_M+n$ and note that (\ref{eq:5}), (\ref{eq:ker})--(\ref{eq:Kseriestil}), imply:
\begin{equation}
\begin{alignedat}{2}\nonumber
f^{\prime}_{0}(z)-\widetilde{\pi}^{\prime}_{n+1}(z)=&\,
\pi r_{0}^{2}\sum_{j=1}^{\infty}\overline{\widetilde{P}_{j}(z_{0})}\widetilde{P}_{j}(z)
-\{\sum_{j=1}^{m}|\widetilde{P}_{j}(z_{0})|^{2}\}^{-1}
\sum_{j=1}^{m}\overline{\widetilde{P}_{j}(z_{0})}\widetilde{P}_{j}(z)\\
=&\sum_{j=1}^{m}\big[\pi r_{0}^{2}-\{\sum_{j=1}^{m}|\widetilde{P}_{j}(z_{0})|^{2}\}^{-1}\big]
\overline{\widetilde{P}_{j}(z_{0})}\widetilde{P}_{j}(z)\\
+&\pi r_{0}^{2}\sum_{j=m+1}^{\infty}\overline{\widetilde{P}_{j}(z_{0})}\widetilde{P}_{j}(z).
\end{alignedat}
\end{equation}
Therefore, by using the orthonormality of $\widetilde{P}_{j}$ we see that,
$$
\|f_{0}^{\prime}-\widetilde{\pi}_{n+1}^{\prime}\|_{L^{2}(G)}^{2}=
\sum_{j=1}^{m}\big[\pi r_{0}^{2}-1/\widetilde{K}_{n}(z_{0},z_{0})\big]^{2}|
\widetilde{P}_{j}(z_{0})|^{2}+(\pi r_{0}^{2})^{2}\sum_{j=m+1}^{\infty}|\widetilde{P}_{j}(z_{0})|^{2}.
$$
Now, using once more (\ref{eq:5}), we obtain, after some trivial calculation, that
$$
\|f_{0}^{\prime}-\widetilde{\pi}_{n+1}^{\prime}\|_{L^{2}(G)}^{2}=%1/K_{n}(z_{0},z_{0})-1/K(z_{0},z_{0})
\left[K(z_{0},z_{0})-\widetilde{K}_{n}(z_{0},z_{0})\right]
\left[K(z_{0},z_{0})\,\widetilde{K}_{n}(z_{0},z_{0})\right]^{-1}.
$$
This and (\ref{eq:3}), with $g(\cdot)=K(\cdot,z_{0})-\widetilde{K}_{n}(\cdot,z_{0})$, leads to
\begin{equation}\label{eq:12}
\|f_{0}^{\prime}-\widetilde{\pi}_{n+1}^{\prime}\|_{L^{2}(G)}^{2}=
\frac{\|K(\cdot,z_{0})-\widetilde{K}_{n}(\cdot,z_{0})\|_{L^{2}(G)}^{2}}{K(z_{0},z_{0})\,
\widetilde{K}_{n}(z_{0},z_{0})}
\end{equation}
and the result (\ref{eq:st1}) follows from the set of the obvious inequalities,
$$
|\widetilde{P}_{1}(z_{0})|=\widetilde{K}_{1}(z_{0},z_{0})\leq\widetilde{K}_{n}(z_{0},z_{0})
\leq K(z_{0},z_{0})=(\pi r_0^2)^{-1},
$$
with constants depending on $r_0$ and $|\widetilde{P}_{1}(z_{0})|$ only.
\end{proof}
\begin{remark}\label{rem:lemma2.1}
It is clear from the proof that the result of Lemma~\ref{l:1} holds true for \textit{any} complete orthonormal system. We note that for the system $\{\widetilde{P}_{j}\}_{j=1}^\infty$ to be complete it suffices that $\Gamma$
is a bounded Jordan curve. In particular, (\ref{eq:st1}) holds with $\pi_{n+1}$ and $K_n$  in the place of $\widetilde{\pi}_{n+1}$ and $\widetilde{K}_n$.
\end{remark}

\section{BKM/AB with pole singularities}\label{sec:3}
In this section we study the BKM and BKM/AB errors
$\|f_{0}-{\pi}_{n}\|_{L^{\infty}(\overline{G})}$ and $\|f_{0}-\widetilde{\pi}_{n}\|_{L^{\infty}(\overline{G})}$,
under the assumption that $f_0$  \textit{has an analytic continuation across} $\Gamma$ in $\Omega$ and its only singularities are poles, or rational poles, of the type (\ref{eq:papam}). More precisely, we refine the classical result (\ref{error1}) for the BKM error, and at the same time we obtain a lower estimate for it. Furthermore, we establish upper and lower estimates for the BKM/AB error. The lower estimates and the refinement are obtained by exploiting the assumption regarding the singularities of $f_0$ and by using certain important results of E.B.\ Saff on polynomial interpolation of meromorphic functions \cite{Sa69}. Since the results of \cite{Sa69} were established for domains with smooth boundaries, we show in the next lemma that they hold true for domains with corners.

In order to do so, we use the \textit{Faber polynomials}  $\{F_n\}_{n=0}^\infty$
of $\overline{G}$. We recall that $F_n(z)$ is defined as the polynomial part of the Laurent series expansion
of $\Phi^n$ at infinity, i.e.,
\begin{equation}\label{faber-def}
F_{n}(z)-\Phi^{n}(z)=O\bigg(\frac{1}{z}\bigg),\mbox{ as }z\rightarrow\infty.
\end{equation}
This, in view of (\ref{eq:Phiatinf}), gives $F_n\in\mathbb{P}_n$ and
\begin{equation}
F_{n}(z)=\gamma^nz^{n}+\cdots.
\end{equation}

Let $L_R$ ($R\ge 1$), denote the level curve of index $R$ of $\Phi$,  i.e.,
\begin{equation}\label{eq:LR}
L_{R}:=\{z:|\Phi(z)|=R\},
\end{equation}
so that $L_1\equiv\Gamma$. Note that $L_R$, for $R>1$, is an analytic Jordan curve. We use $G_R$ to denote its interior, i.e.,
$
G_R:=\textrm{int}(L_R).
$
The following result gives the exact rate of convergence of the minimum uniform error in approximating  meromorphic functions by polynomials.
\begin{lemma}\label{l:2}
Assume that the boundary $\Gamma$ of $G$ is piecewise Dini-smooth and consider a function $f$ which is analytic on
$\overline{G_\varrho}$, for some $\varrho>1$, apart from a finite number of poles on $L_\varrho$. Let $m$ denote the highest order of the poles of $f$ on $L_\varrho$. Then,
\begin{equation}\label{eq:15}
\inf_{p\in\mathbb{P}_{n}}\|f-p\|_{L^{\infty}(\overline{G})}
\asymp\frac{n^{m-1}}{\varrho^{n}}.
\end{equation}
\end{lemma}

A curve $\Gamma$ is piecewise Dini-smooth if it consists of a finite number of Dini-smooth arcs.
An arc $z=z(s)$, where $s\in[a,b]$ stands for the arclength, is called \textit{Dini-smooth} if $z^{\prime}(s)$ is continuous on $[a,b]$, and if $z^{\prime}(s)$ has a modulus of continuity $\omega$ which satisfies
$\int_{0}^{\alpha}\left[{\omega(t)}/{t}\right]\,dt<\infty$, for some $\alpha>0$.
We note, in particular, that a piecewise Dini-smooth curve may have corners or cusps and
that a piecewise analytic Jordan curve is also piecewise Dini-smooth.
\begin{proof}
We recall the following two facts regarding Faber polynomials:
\begin{itemize}
\item[(i)]
For any $r$, $R$, with $1<r<R$, it holds
\begin{equation}\label{eq:suetin}
F_{n}(z)=\Phi^{n}(z)\bigg\{1+O\bigg(\frac{r^{n}}{R^{n}}\bigg)\bigg\},\quad z\in L_{R},
\end{equation}
see e.g.\ \cite[p.~43]{Su98}.
\item[(ii)]
Under the assumption on $\Gamma$, the Faber polynomials are uniformly bounded on $\overline{G}$ (see \cite{Ga99}), i.e.,
\begin{equation}\label{eq:M8}
\|F_n\|_{L^{\infty}(\overline{G})}\leq c(\Gamma),\quad n\in\mathbb{N},
\end{equation}
where $c(\Gamma)$ is a positive constant that depends on $\Gamma$ only.
\end{itemize}

Observe that  $(\ref{eq:suetin})$ implies that the sequence $\{F_{n}(z)\}_{n=1}^{\infty}$ has no limit point of zeros exterior to $\overline{G}$.
Also, from $(\ref{eq:suetin})$ and $(\ref{eq:M8})$ we have for $z\in\overline{G}$ and $t\in L_\varrho$ that,
\begin{equation}
\frac{|F_{n}(z)|}{|F_{n}(t)|}\le\frac{c_1(\Gamma)}{\varrho^{n}}, \quad n\in\mathbb{N}.
\end{equation}
Now, following the proof of Theorem $2$ of \cite{Sa69} and using the sequence of the Faber polynomials $\{F_{n}\}$ in the place of $\{\omega_{n}\}$, we conclude that there exist polynomials $\{p_{n}\}_{n=1}^\infty$, such that
\begin{equation}
\|f-p_{n}\|_{L^{\infty}(\overline{G})}\leq c_2(\Gamma)\frac{n^{m-1}}{\varrho^{n}}, \quad n\in\mathbb{N},
\end{equation}
see also \cite[p.~399]{SaSt08}.
This yields the upper bound in (\ref{eq:15}). The lower bound follows at once from Theorem~10 of \cite{Sa69}, by observing that
$\Omega$ is simply-connected and hence its Green function with pole at infinity has no critical points.
\end{proof}

The following result is the so-called Andrievskii's lemma for polynomials and rational polynomials.
Its proof, for bounded Jordan domains such that the inverse conformal map $g:\mathbb{D}\to G$ satisfies a Lipschitz condition on $\overline{\mathbb{D}}$,  can be found in \cite{Ga87}. This condition is certainly satisfied by the type of domains considered below.
\begin{lemma}\label{lem:andri-gaier}
Assume that $\Gamma$ is piecewise analytic without cusps. Then:
\begin{itemize}
\item[(i)]
For any $P_n\in\mathbb{P}_n$, with $P_n(z_0)=0$, it holds
\begin{equation}\label{eq:andri-lemma}
\|P_n\|_{L^{\infty}(\overline{G})}\ole \sqrt{\log n}\,\|P_n^\prime\|_{L^{2}({G})},\quad n\ge 2.
\end{equation}
\item[(ii)]
For any $P_n\in\mathbb{P}_n$, with $P_n(z_0)=0$, and $q$ a fixed polynomial with no zeros on $\overline{G}$, it holds that
\begin{equation}\label{eq:andri-gaier-lemma}
\|P_n/q\|_{L^{\infty}(\overline{G})}\ole \sqrt{\log n}\,\|(P_n/q)^\prime\|_{L^{2}({G})},\quad n\ge 2.
\end{equation}
\end{itemize}
\end{lemma}

\subsection{BKM}
The next theorem complements  the classical result (\ref{error1}) of Walsh, in the sense that it provides a lower estimate and, in addition, uses the precise $\varrho=|\Phi(z_1)|$ in the denominator, instead of any $R$, with $1<R<\varrho$. This is done by utilizing extra information on the nature of the singularities of $f_0$ in $\Omega$.
\begin{theorem}\label{th:1}
Assume that $\Gamma$ is piecewise analytic without cusps. Assume further that the conformal map $f_{0}$ has an analytic continuation across $\Gamma$, such that $f_0$ is analytic on $\overline{G_\varrho}$, for some $\varrho>1$, apart from a finite number of poles on $L_\varrho$. Let $m$ denote the highest order of the poles of $f_0$ on $L_\varrho$. Then,
\begin{equation}\label{eq:st2}
\frac{n^{m-1}}{\varrho^{n}}
\preceq\|f_{0}-\pi_{n}\|_{L^{\infty}(\overline{G})}
\preceq\frac{n^{m}\sqrt{\log n}}{\varrho^{n}},\quad n\ge 2.
\end{equation}
\end{theorem}
\begin{proof}
We observe first that the kernel $K(z,z_0)$ shares the same analytic properties with $f_0$ on $\overline{G_\varrho}$, apart from an unit increase on the order of its poles on $L_\varrho$.
Therefore, using Lemma \ref{l:2} with $f\equiv K(\cdot,z_{0})$, we conclude that
\begin{equation}\label{eq:m1}
\|K(\cdot,z_0)-p_{n}\|_{L^{\infty}(\overline{G})}\ole \frac{n^{m}}{\varrho^{n}},\quad n\in\mathbb{N},
\end{equation}
for some sequence of polynomials $\{p_{n}\}_{n=1}^{\infty}$.
Since the $L^{2}(G)$-norm is dominated by the $L^{\infty}(\overline{G})$-norm, (\ref{eq:m1}) leads to the estimate
\begin{equation}
\|K(\cdot,z_{0})-p_{n}\|_{L^{2}(G)}\ole\frac{n^{m}}{\varrho^{n}}.
\end{equation}
Then, the minimum property of the kernel polynomials implies that
\begin{equation}\label{eq:kernel1}
\|K(\cdot,z_{0})-K_{n}(\cdot,z_{0})\|_{L^{2}(G)}\ole\frac{n^{m}}{\varrho^{n}},
\end{equation}
which, in conjunction with Remark~\ref{rem:lemma2.1}, yields the estimate
\begin{equation}
\|f_{0}^{\prime}-\pi_{n}^{\prime}\|_{L^{2}({G})}\ole\frac{n^{m}}{\varrho^{n}}.
\end{equation}

Now, we use Andrievskii's Lemma~\ref{lem:andri-gaier}(i) and employ the method of Andrievskii and Simonenko,
see e.g.\ \cite[\S2.1]{Ga88}. This method enables the transition from an upper bound of the error
$\|f_{0}^{\prime}-\pi_{n}^{\prime}\|_{L^{2}({G})}$ to a similar bound for the error
$\|f_{0}-\pi_{n}\|_{L^{\infty}(\overline{G})}$, with the extra cost of a $\sqrt{\log n}$ factor, and leads to
the upper estimate in (\ref{eq:st2}). The lower estimate follows immediately from Lemma~\ref{l:2}.
\end{proof}

The following pointwise estimate is useful in the study of the distribution of the zeros of the Bergman polynomials;
see e.g., \cite{LSS}, \cite{M-DSS}, \cite{SaSt08} and \cite{GPSS}.
\begin{corollary}\label{cor:thm3.1}
With the assumptions of Theorem~\ref{th:1} it holds,
\begin{equation}
|P_n(z_{0})|\preceq\frac{n^{m}}{\varrho^n},\quad n\in\mathbb{N}, \quad (z_0\in G).
\end{equation}
\end{corollary}
\begin{proof}
The result emerges easily from (\ref{eq:kernel1}),  using the reproducing property of $K(\cdot,z_{0})$, the fact that
$P_{n+1}$ is orthogonal to any polynomial in $\mathbb{P}_n$, and the Cauchy-Schwarz inequality:
\begin{equation*}
\begin{alignedat}{3}
|{P}_{n+1}(z_{0})|
&=|\langle {P}_{n+1},K(\cdot,z_{0})\rangle|=|\langle {P}_{n+1},K(\cdot,z_{0})-K_n(\cdot,z_0)\rangle| \\
&\le \|{P}_{n+1}\|_{L^2(G)}\,\|K(\cdot,z_{0})-K_{n}(\cdot,z_0)\|_{L^2(G)}\\
&=\|K(\cdot,z_0)-K_{n}(\cdot,z_0)\|_{L^2(G)}.
\end{alignedat}
\end{equation*}
\end{proof}

\subsection{BKM/AB with pole singularities}\label{ss:pole}
We exploit now the specific assumptions on the singularities of the analytic extension of $f_0$ studied in Section~\ref{subsec:poles}.
More precisely, the assumption that the nearest singularities of $f_{0}$ are $\kappa$ poles, each one of order $k_{j}$ at $z_{j}$, $j=1,2,\ldots \kappa$, where $|\Phi(z_{1})|\le|\Phi(z_{2})|\le\cdots\le|\Phi(z_{\kappa})|$, and that the other singularities of $f_{0}$ occur at points
$z_{\kappa+1},z_{\kappa+2},\ldots$, where
$|\Phi(z_{\kappa})|<|\Phi(z_{\kappa+1})|\leqslant|\Phi(z_{\kappa+2})|\leqslant\cdots$. Therefore, for the BKM/AB we consider the system $\{\eta_{j}\}$, defined  by the singular functions in (\ref{eq:7}), with
$m_j=1$, $j=1,2,\ldots,\kappa$, and the $n$ monomials in (\ref{eq:9}). Accordingly, we let $\mathbb{P}_n^{A_1}$
denote the following space of augmented polynomials:
\begin{equation}
\mathbb{P}_{n}^{A_{1}}:=\{p:p(z)=\sum_{j=1}^{\kappa+n}t_j\eta_j(z),\,\, t_{j}\in\mathbb{C}\}.
\end{equation}
We note that the associated augmented kernel polynomial $\widetilde{K}_{n}(z,z_{0})$ is
the best approximation to $K(z,z_{0})$ in $L^{2}(G)$ out of the
space $\mathbb{P}_{n}^{A_{1}}$, i.e.,
\begin{equation}\label{eq:min}
\|K(\cdot,z_{0})-\widetilde{K}_{n}(\cdot,z_{0})\|_{L^{2}(G)}\leq\|K(\cdot,z_{0})-p\|_{L^{2}(G)},
\end{equation}
for any $p\in \mathbb{P}_{n}^{A_{1}}$.

The next theorem provides an estimate for the error in the resulting BKM/AB approximation
$\widetilde{\pi}_{n}$ to $f_0$.
\begin{theorem}\label{th2}
Assume that $\Gamma$ is piecewise analytic without cusps and set $\varrho:=|\Phi(z_{\kappa+1})|$.
Then,
\begin{equation}\label{eq:th:2}
\|f_{0}-\widetilde{\pi}_{n}\|_{L^{\infty}(\overline{G})}\ole\frac{1}{R^{n}},
\end{equation}
for any $R$, with $1<R<\varrho$, but for no $R>\varrho$.
\end{theorem}
\begin{proof}
Observe that $K(\cdot,z_0)$ has poles of order $k_{j}+1$ at each $z_{j}, j=1,2,\ldots,\kappa$, and set $Q(z):=\prod_{j=1}^{k}(z-z_{j})^{k_{j}+1}$. Then, the function $K(z,z_0)Q(z)$ is analytic in the interior $G_\varrho$
of the level curve of $L_\varrho$, and from Walsh's maximal convergence theorem \cite[\S4.7]{Wa} it follows that,
for any $R$, with $1<R<\varrho$, there exists a sequence
of polynomial $\{p_n\}_{n=1}^{\infty}$, such that
\begin{equation}\label{eq:Walshfp}
\|K(\cdot,z_0)Q-p_n\|_{L^{\infty}(\overline{G})}\ole\frac{1}{R^n},\quad n\in\mathbb{N}.
\end{equation}

Let now $d:=\min_{j=1,2,\ldots,\kappa}\{|z-z_{j}|:\,z\in \Gamma\}$ denote the distance of $\Gamma$ from the poles
$\{z_{j}\}_{j=1}^\kappa$, and set $\xi:=\sum_{j=1}^{\kappa}k_{j}$.
Then, $|Q(z)|\geq d^{\kappa+\xi}$, $z\in\Gamma$, and (\ref{eq:Walshfp}) gives
$$
\|K(\cdot,z_0)-\frac{p_{n}}{Q}\|_{L^\infty(\overline{G})}\le\frac{c}{d^{\kappa+\xi}}\frac{1}{R^n}.
$$
Since the $L^{2}(G)$-norm is dominated by the $L^{\infty}(\overline{G})$-norm, we see that there exist a sequence of rational polynomials $\{Q_{n}\}_{n=1}^{\infty}$, with $Q_{n}\in \mathbb{P}_{n}^{A_{1}}$, such that,
\begin{equation*}
\|K(\cdot,z_0)-Q_{n}\|_{L^{2}(G)}\ole \frac{1}{R^n},\quad n\in\mathbb{N}.
\end{equation*}
Therefore, using the minimum property (\ref{eq:min}) of the augmented kernel polynomials,  we have
\begin{equation}\label{eq:minKt}
\|K(\cdot,z_0)-\widetilde{K}_n(\cdot,z_0)\|_{L^{2}(G)}\ole \frac{1}{R^n},\quad n\in\mathbb{N},
\end{equation}
and this, in conjunction with the equivalence Lemma \ref{l:1}, yields the estimate
\begin{equation}
\|f_{0}^{\prime}-\widetilde{\pi}_{n}^{\prime}\|_{L^{2}(G)}\ole \frac{1}{R^n},\quad n\in\mathbb{N}.
\end{equation}

Next, we recall that
\begin{equation}\label{eq:widepin}
\widetilde{\pi}_n(z)=
\sum_{j=1}^\kappa c_{n,j}\left[\frac{1}{(z-z_j)^{k_j}}-\frac{1}{(z_0-z_j)^{k_j}}\right]
+\sum_{j=1}^{n} c_{n,\kappa+j}\left[z^j-z_0^j\right],
\end{equation}
i.e.,
\begin{equation}\label{eq:widepin1}
\widetilde{\pi}_n(z)=\frac{P(z)}{q(z)}, \mbox{ where }\, q(z):=\prod_{j=1}^\kappa(z-z_j)^{k_j},
\end{equation}
and $P(z)$ is a polynomial of degree $n+\xi$.

Then, the transition from the $L^2(G)$-norm to the $L^\infty(\overline{G})$-norm is done as in the proof of Theorem~\ref{th:1},
where now, in view of (\ref{eq:widepin1}), Lemma~\ref{lem:andri-gaier}(ii) is applicable.  This leads to,
\begin{equation}
\|f_0-\widetilde{\pi}_n\|_{L^\infty(\overline{G})}
\ole \frac{\sqrt{\log n}}{R^n},\quad n\ge 2,
\end{equation}
and (\ref{eq:th:2}) follows with a different, but still less than $\varrho$,  $R$.

Finally, the fact that (\ref{eq:th:2}) holds for no $R>\varrho$ is evident from \cite[Thm.~6, Ch.~IV]{Wa}, since the contrary assumption would lead to the contradicting conclusion that $f_0$ has no singularities on $L_\varrho$; see the next remark.
\end{proof}

\begin{remark}\label{rem:PapWa}
From (\ref{eq:widepin}) it is clear that $\widetilde{\pi}_n(z)=\widetilde{q}_\kappa(z)+p_n(z)$, where  $\widetilde{q}_\kappa$ is defined by the
nearest $\kappa$ poles of $f_0$ in $\Omega$ and $p_n\in\mathbb{P}_n$. Hence, (\ref{eq:th:2}) gives
\begin{equation*}%\label{eq:th:2}
\|(f_{0}-\widetilde{q}_\kappa)-p_n\|_{L^{\infty}(\overline{G})}
\ole\frac{1}{R^{n}},
\end{equation*}
for any $1<R<\varrho$, and Theorem~6 in \cite[Ch.~V]{Wa} implies that the function $f_{0}-\widetilde{q}_\kappa$
is analytic in $G_\varrho$. This shows that the the rational polynomial $\widetilde{q}_\kappa$, constructed
by the BKM/AB considered above, \textit{cancels out the specific poles of $f_0$ that contains}. In particular,
this provides the theoretical justification for the heuristic observation made to that effect by Papamichael and Warby in \cite[p.~652]{PW86}.
\end{remark}
%We note that, in the special case where all the poles $z_1,z_2,\ldots,z_\kappa$, are simple, and such that %$|\Phi(z_{1})|=|\Phi(z_{2})|=\ldots=|\Phi(z_{\kappa})|$, the result of Theorem~\ref{th2}
%was indicated in \cite[pp.~651--652]{PW86}.

A finer estimate than (\ref{eq:th:2}) can be obtained if the singularities of $f_0$ on $L_{|\Phi(z_{\kappa+1})|}$
are a finite number of poles.
\begin{theorem}\label{th2ii}
Assume that $\Gamma$ is piecewise analytic without cusps and set $\varrho:=|\Phi(z_{\kappa+1})|$.
Assume, in addition to Theorem~\ref{th2}, that  $f_0$  has a finite number of poles and no other singularities on
$L_\varrho$ and let $m$ denote their highest order. Then,
\begin{equation}
\frac{n^{m-1}}{\varrho^{n}}
\preceq\|f_{0}-\widetilde{\pi}_{n}\|_{L^{\infty}(\overline{G})}
\preceq\frac{n^{m}\sqrt{\log n}}{\varrho^{n}},\quad n\ge 2.
\end{equation}
\end{theorem}
\begin{proof}
The upper estimate follows by working in the same way as in the proof of Theorem~\ref{th2}, but using here the
precise result of Lemma~\ref{l:2}, in the place of Walsh's theorem in (\ref{eq:Walshfp}).

To obtain the lower estimate, observe that
$q\widetilde{\pi}_n$ is a polynomial of degree $n+\xi$ (see (\ref{eq:widepin1})) and that the function $qf_{0}$ is analytic on $\overline{G_\varrho}$,
apart from a finite number of poles on $L_\varrho$. Hence, from Lemma \ref{l:2} we have, for $n\in\mathbb{N}$,
\begin{equation}
\|qf_0-q\widetilde{\pi}_n\|_{L^{\infty}(\overline{G})}
\succeq\inf_{p\in\mathbb{P}_{n+\xi}}\|qf_0-p\|_{L^{\infty}(\overline{G})}
\succeq\frac{n^{m-1}}{\varrho^n},
\end{equation}
which yields the estimate
\begin{equation*}
\|f_0-\widetilde{\pi}_n\|_{L^\infty(\overline{G})}\ge
\frac{c}{\|q\|_{L^\infty(\overline{G})}}\frac{n^{m-1}}{\varrho^n}.
\end{equation*}
and hence the required result.
\end{proof}
%where $D:=\max_{j=1,2,\ldots,\kappa}\{|z-z_{j}|:\,z\in \Gamma\}$, and $S:=\sum_{j=1}^{\kappa}k_{j}$.

In the more general case, where the nearest $\kappa$ singularities of $f_0$ in $\Omega$ are rational poles of the type (\ref{eq:papam}), we have the following result regarding the associated kernel polynomials $\widetilde{K}_{n}(\cdot,z_{0})$.
\begin{theorem}\label{th:7}
Assume that $\Gamma$ is piecewise analytic without cusps and set $\varrho:=|\Phi(z_{\kappa+1})|$.
Then,
\begin{equation}\label{eq:28}
\|K(\cdot,z_{0})-\widetilde{K}_{n}(\cdot,z_{0})\|_{L^{2}(G)}\ole\frac{1}{R^{n}},
\end{equation}
for any $R$, $1<R<\varrho$.
\end{theorem}
\begin{proof}
Set $Q(z):=\prod_{j=1}^{k}(z-z_{j})^{{k_{j}}/{m_{j}}+1}$ and follow the proof of Theorem~\ref{th2} up to
Equation (\ref{eq:minKt}).
\end{proof}

\section{BKM with pole and corner singularities}\label{sec:4}
In this section we assume that $f_0$  \textit{has a singularity on} $\Gamma$ and
study the BKM and BKM/AB errors, corresponding to a variety of different syntheses of the system $\{\eta_j\}$
of basis functions. In stating the results we use the notation and the assumptions set up in
Sections~\ref{subsec:corners} and \ref{subsec:poles}.

\subsection{BKM}
Our first result is a straightforward consequence of Theorem~3.1 of \cite{SaSt08} and
Lemma~\ref{lem:andri-gaier} above.
\begin{theorem}\label{th3}
Assume that $\Gamma$ is piecewise analytic without cusps and set $\varrho:=|\Phi(z_1)|$ and
$s:=\min\{(2-\alpha_k)/\alpha_k:\, 1\leq k\le M\}$. Then,
\begin{equation}\label{eq:th3}
\|f_{0}-\pi_n\|_{L^{\infty}(\overline{G})}\leq
c_{1}\sqrt{\log n}\frac{1}{n^{s}}+c_{2}\frac{1}{R^{n}},\quad n\ge2.
\end{equation}
for any $R$, $1<R<\varrho$.
\end{theorem}
\begin{proof}
Observe that $\zeta_1$ in Theorem~3.1 of \cite{SaSt08} can be chosen arbitrarily close to $z_1$. Thus, from the minimum
property of the kernel polynomials $K_n(\cdot,z_0)$ we have
\begin{equation}\label{eq:SSest}
\|K(\cdot,z_0)-K_n(\cdot,z_0)\|_{L^{2}(G)}\leq
c_1\frac{1}{n^{s}}+c_2\frac{1}{R^{n}},\quad n\in\mathbb{N},
\end{equation}
for any $R$, $1<R<\varrho$, and the transition from the $L^2(G)$-error in (\ref{eq:SSest}) to the $L^\infty(\overline{G})$-error in (\ref{eq:th3}), goes in the same lines as in the proof of Theorem~\ref{th2}.
\end{proof}

\begin{remark}
Clearly, as $n\to\infty$, (\ref{eq:th3}) yields the result (\ref{eq:6}). However, Theorem~\ref{th3} does more:
It captures, in a very precise form, the dependance of the BKM error $\|f_{0}-\pi_{n}\|_{L^{\infty}(\overline{G})}$
for \lq\lq small" values of $n$, on both the corner and pole singularities of $f_0$. This dependance has been testified numerically in \cite{LPS} and has given rise to the introduction of the BKM/AB.
\end{remark}

The following result is a simple consequence of (\ref{eq:SSest}). Its proof is similar to that of Corollary~\ref{cor:thm3.1}.
\begin{corollary}\label{cor:th3}
With the assumptions of Theorem~\ref{th3} it holds,
\begin{equation}\label{eq:SSestPn}
|P_n(z_{0})|\le c_1\frac{1}{n^{s}}+c_2\frac{1}{R^{n}},\quad n\in\mathbb{N}, \quad (z_0\in G).
\end{equation}
\end{corollary}

\begin{remark}
Since $|P_n(z_0)|\leq\|K(\cdot,z_0)-K_n(\cdot,z_0)\|_{L^{2}(G)}$, it follows from Corollary~\ref{cor:th3} that, if for small values of $n$, $P_n(z_0)$ decays geometrically to zero, then the most \lq\lq serious" singularity of $K(\cdot,z_0)$, and hence of $f_0$, is the nearest pole in $\Omega$ and not an algebraic singularity on the boundary, as the asymptotic estimate (\ref{eq:6}) would suggest.
On the other hand, given that $f_0$ has a singularity on $\Gamma$, Theorem~2.1 of \cite{LSS} implies that any point of $\Gamma$ is a point of accumulation of the zeros of the sequence $\{P_n\}_{n=1}^\infty$. Therefore, an easy way to check whether a pole singularity is more serious than an algebraic singularity, for a range of values of $n$, is by plotting the zeros of $P_n$ for the same range: \textit{If the zeros stay away from a specific part of the boundary, this indicates that $P_n(z_0)$ decays geometrically and therefore the presence of a pole singularity near that part}.
We refer to \cite[Examples 2, 3]{SaSt08}, where (\ref{eq:SSestPn}) was used as the tool for explaining the misleading nature of such plots.
\end{remark}

\subsection{BKM/AB with corner singularities}\label{ss:corner}
%Now we follow closely the construction in \cite[\S2.5]{AG}.
From our assumptions on $\Gamma$, it follows that the conformal map $f_0$ can be extended analytically, by means of the reflection principle, beyond $\Gamma$ to a larger Jordan domain $\widetilde{G}$, such that the boundary $\partial\widetilde{G}$ of $\widetilde{G}$ consists of  analytic arcs to be fixed below.
For this, we recall our assumptions on the position of the nearest poles $z_j$, $j=1,\ldots,\kappa$, of $f_0$ in $\Omega$ and pick up a point $\zeta_1$ near $z_1$, but interior to the level curve  $L_\varrho$, with $\varrho:=|\Phi(z_1)|$.
Next, we draw the level curve $L_{\widetilde{\varrho}}$, with $\widetilde{\varrho}:=|\Phi(\zeta_1)|$ and fix on it
points $\zeta_k$, $k=2,\ldots,N$, "between" $\tau_k$ and $\tau_{k+1}$, where we set $\tau_{N+1}=\tau_1$. We connect
each non-special corner $\tau_k$, $k=1,\ldots,M$, with the two adjacent $\zeta_k$'s, by using two analytic arc.
Next, we denote by $l_k$ the two arcs emanating from $\tau_k$ and call $l_N$ the part (or parts) of the level line
$L_{\widetilde{\varrho}}$ that joins together those consecutive points $\zeta_k$ that have only one connection with $\tau_k$.
See Figure~\ref{fig:tilG}, for a possible arrangement of corners $\tau_k$, points $\zeta_k$, and arcs $l_k$ and $l_N$.
Finally, we define $\widetilde{G}$ by taking $\partial\widetilde{G}:=\{\cup_{k=1}^Ml_k\}\cup l_N$.

The above construction is such that:
\begin{itemize}
\item[(i)]
$\partial\widetilde{G}$ is a piecewise analytic Jordan curve that
meets $\Gamma$ at the  non-special corner $\tau_k$, $k=1,\ldots,M$.
\item[(ii)]
$f_0$ is continuous on $\widetilde{G}\cup\partial\widetilde{G}$ and analytic in $\widetilde{G}$ and on $\partial\widetilde{G}$, except for the endpoints $\tau_k$.
\item[(iii)]
The asymptotic expansion (\ref{eq:lemman2}) holds for $z\in l_k$, $k=1,\ldots,M$, in the sense that, for any $p_k\in\mathbb{N}_0$,
\begin{equation}\label{eq:lehmanO}
\begin{alignedat}{2}
&f_{0}(z)=\sum_{j=0}^{p_k}a_{j}^{(k)}(z-\tau_{k})^{\gamma_{j}^{(k)}}
+\widetilde{f}_{\gamma_{p_k+1}^{(k)},\tau_{k}}(z),\\
&\widetilde{f}_{\gamma_{p_k+1}^{(k)},\tau_{k}}(z)=O\left((z-\tau_{k})^{\gamma_{p_k+1}^{(k)}}\right).
\end{alignedat}
\end{equation}
\end{itemize}

\begin{figure}[t]
\begin{center}
\includegraphics*[scale=0.9]{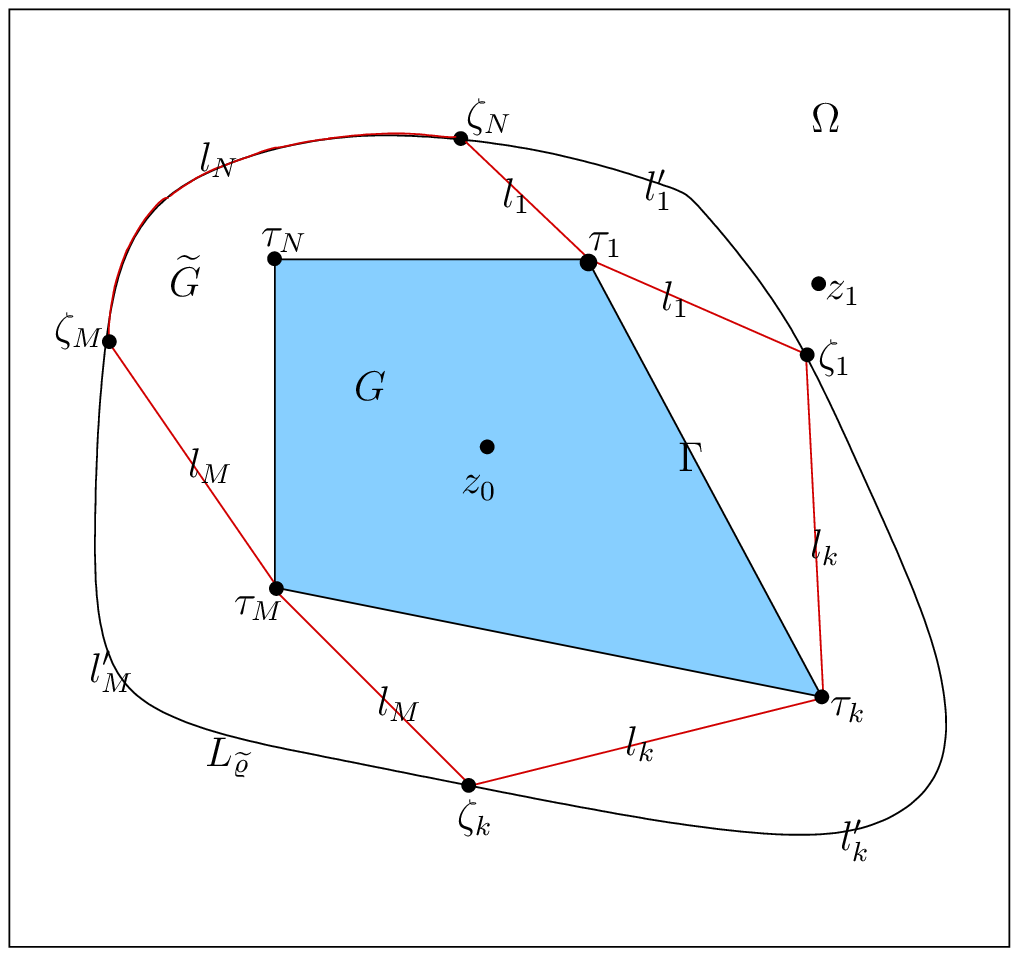}
\caption{The domain $\widetilde{G}$ in the proof of Theorem \ref{th:5}.}
\label{fig:tilG}
\end{center}
\end{figure}

We consider now the application of BKM/AB with only corner singularities, where we use  $p_{k}\in\mathbb{N}_{0}$
singular function for each non-special corner $\tau_k$, $k=1,2,\ldots,M$. In order to measure the BKM/AB error we set
\begin{equation}\label{eq:nuk}
\nu_{k}:=\min\{j>p_{k}:\ \gamma_{j}^{(k)}\notin\mathbb{N},\ a_{j}^{(k)}\neq0\},
\end{equation}
and assume that at least one of $\nu_k$'s is finite, otherwise the results become trivial.
The associated BKM/AB system $\{\eta_j\}$ is thus defined by $r_M=\sum_{k=1}^Mp_k$ singular functions of the form (\ref{eq:8})
and $n$ monomials (\ref{eq:9}). Accordingly, we let $\mathbb{P}_n^{A_{2}}$ denote the space of augmented polynomials:
\begin{equation}
\mathbb{P}_{n}^{A_{2}}:=\{p:p(z)=\sum_{j=1}^{r_M+n}t_j\eta_j(z),\,\, t_j\in\mathbb{C}\}.
\end{equation}
Clearly, the associated augmented polynomial $\widetilde{K}_{n}(z,z_{0})$ is
the best approximation to $K(z,z_{0})$ in $L^{2}(G)$ out of the space $\mathbb{P}^{A_{2}}_{n}$.

Let $\widetilde{\pi}_n$ denote the BKM/AB approximation resulting from $\mathbb{P}_{n}^{A_{2}}$.
Then we have the following:
\begin{theorem}\label{th:5}
Assume that $\Gamma$ is piecewise analytic without cusps and set $\varrho:=|\Phi(z_1)|$ and
$s^\star:=\min\{(2-\alpha_k)\gamma_{\nu_k}^{(k)}:\ 1\leq k\le M\}$. Then,
\begin{equation}\label{eq:27}
\|f_{0}-\widetilde{\pi}_{n}\|_{L^{\infty}(\overline{G})}\leq
c_{1}\sqrt{\log n}\frac{1}{n^{s^\star}}+c_{2}\frac{1}{R^{n}},\quad n\ge2,
\end{equation}
for any $R$, $1<R<\varrho$.
\end{theorem}

\begin{proof}
Using Cauchy's integral formula for the derivative of the extension of $f_0$ we have, for $z\in G$,
%(see also \cite[p.~396]{SaSt08})
\begin{equation}\label{eq:f0p}
\begin{alignedat}{2}
f_{0}^{\prime}(z)
&=\frac{1}{2\pi i}\int_{\partial\widetilde{G}}\frac{f_{0}(t)}{(t-z)^{2}}dt\\
&=\frac{1}{2\pi i}\sum_{k=1}^{M}\int_{l_k}\frac{f_{0}(t)}{(t-z)^{2}}dt +
\frac{1}{2\pi i}\int_{l_N}\frac{f_{0}(t)}{(t-z)^{2}}dt.
\end{alignedat}
\end{equation}

For each $\tau_k$, $k=1,\ldots,M$, we consider the first terms up to $p_k$, of the Lehman expansion (\ref{eq:lehmanO})
for $f_0$:
\begin{equation}\label{eq:Fk}
F_{k}(z):=\sum_{j=0}^{p_k}a_{j}^{(k)}(z-\tau_{k})^{\gamma_{j}^{(k)}}.
\end{equation}
Since the function $F_k(z)$, is analytic in $\widetilde{G}$ and continuous on $\partial\widetilde{G}$
we have, as in (\ref{eq:f0p}), for $z\in G$,
\begin{equation*}
F_k^{\prime}(z)=\frac{1}{2\pi i}\sum_{r=1}^{M}\int_{l_r}\frac{F_{k}(t)}{(t-z)^{2}}dt +
\frac{1}{2\pi i}\int_{l_N}\frac{F_{k}(t)}{(t-z)^{2}}dt.
\end{equation*}
Therefore,
\begin{equation*}
\begin{alignedat}{2}
\sum_{k=1}^M F_k^{\prime}(z)
&=\frac{1}{2\pi i}\sum_{k=1}^M\sum_{r=1}^{M}\int_{l_r}\frac{F_{k}(t)}{(t-z)^{2}}dt +
\frac{1}{2\pi i}\sum_{k=1}^M\int_{l_N}\frac{F_{k}(t)}{(t-z)^{2}}dt\\
&=\frac{1}{2\pi i}\sum_{k=1}^M\int_{l_k}\frac{F_{k}(t)}{(t-z)^{2}}dt
+\frac{1}{2\pi i}\sum_{k=1}^M\sum_{r=1\atop{r\neq k}}^{M}\int_{l_r}\frac{F_{k}(t)}{(t-z)^{2}}dt\\
&+\frac{1}{2\pi i}\sum_{k=1}^M\int_{l_N}\frac{F_{k}(t)}{(t-z)^{2}}dt.
\end{alignedat}
\end{equation*}
Hence, for $z\in G$,
\begin{equation}\label{eq:sum}
f_{0}^{\prime}(z)-\sum_{k=1}^{M}F_k^{\prime}(z)=g(z)+h(z),
\end{equation}
where,
\begin{equation}\label{eq:gdef}
g(z):=\frac{1}{2\pi i}\sum_{k=1}^{M}\int_{l_{k}}\frac{f_0(t)-F_k(t)}{(t-z)^{2}}dt,
\end{equation}
and
\begin{equation}\label{eq:honlr}
\begin{alignedat}{2}
h(z)
&:=\frac{1}{2\pi i}\int_{l_N}\frac{f_{0}(t)}{(t-z)^{2}}dt-\frac{1}{2\pi i}
\sum_{k=1}^M\sum_{r=1\atop{r\neq k}}^{M}\int_{l_r}\frac{F_{k}(t)}{(t-z)^{2}}dt  \\
&-\frac{1}{2\pi i}\sum_{k=1}^M\int_{l_N}\frac{F_{k}(t)}{(t-z)^{2}}dt.
\end{alignedat}
\end{equation}

Now, we denote by $l^\prime_r$, $r=1,\ldots,M$, the part of the level line $L_{\widetilde{\varrho}}$ that shares the same endpoints with $l_r$, so that $L_{\widetilde{\varrho}}=\{\cup_{r=1}^M l^\prime_r\}\cup l_N$ and $l_r\cup l^\prime_r$ is the boundary of a Jordan domain in $\Omega$; see  Figure~\ref{fig:tilG}.
Since, for $k\neq r$, the function $F_k(z)$, $z\in G$, is analytic in the interior of $l_r\cup l^\prime_r$ and
continuous on $l_r\cup l^\prime_r$,  we can replace in (\ref{eq:honlr}) the path of integration $l_r$ by $l^\prime_r$, with suitable orientation, i.e., for $z\in G$,
\begin{equation}\label{eq:honlpr}
\begin{alignedat}{2}
h(z)&=\frac{1}{2\pi i}\int_{l_N}\frac{f_{0}(t)}{(t-z)^{2}}dt-\frac{1}{2\pi i}
\sum_{k=1}^M\sum_{r=1\atop{r\neq k}}^{M}\int_{l^\prime_r}\frac{F_{k}(t)}{(t-z)^{2}}dt  \\
&-\frac{1}{2\pi i}\sum_{k=1}^M\int_{l_N}\frac{F_{k}(t)}{(t-z)^{2}}dt.
\end{alignedat}
\end{equation}
Observe that, by construction, $f_0$ is continuous on $l_N$ and $F_k$ is continuous on
$l_N\cup l^\prime_r$, for $k\neq r$. Thus, the function $h$ in (\ref{eq:honlpr}) is analytic in $G_{\widetilde{\varrho}}$ and by Walsh's maximal convergence theorem there exist a sequence of polynomials $\{t_n\}_{n=1}^{\infty}$ such that,
\begin{equation}\label{eq:24}
\|h-t_n\|_{L^{\infty}(\overline{G})}\preceq\frac{1}{R^{n}},\quad n\in\mathbb{N},
\end{equation}
where $1<R<\widetilde{\varrho}$.
Since we can choose $\zeta_1$ arbitrarily close to $z_1$,  (\ref{eq:24}) is valid for any $1<R<\varrho$.

The function $g$ in (\ref{eq:gdef}) consists of sums of integrals of the type,
$$
G(z)=\int_{l_{k}}\frac{g_{k}(t)}{(t-z)^{2}}dt,
$$
where in view of (\ref{eq:lehmanO}) and (\ref{eq:Fk}) we have , for $t\in l_k$,
$$
|g_k(t)|\preceq|t-\tau_{k}|^{\gamma_{p_k+1}^{(k)}}.
$$
Hence, by using the result of Lemma~$6$ in \cite{AG}, in conjunction with the remark following Theorem $1$ of the same paper and the triangle inequality, we conclude that there exists a sequence of polynomials $\{q_{n}\}_{n=1}^{\infty}$
satisfying
\begin{equation}\label{eq:25}
\|g-q_{n}\|_{L^{2}(G)}\preceq\frac{1}{n^{\widetilde{s}}},\quad n\in\mathbb{N},
\end{equation}
where $\widetilde{s}:=\min\{(2-\alpha_{k})\gamma_{p_k+1}^{(k)}:\,\,k=1,2,\ldots,M\}$.
This, combined with (\ref{eq:sum}), (\ref{eq:24}) and the triangle inequality, yields
\begin{equation}\label{eq:f-pn}
\|f_{0}^{\prime}-\sum_{k=1}^{M}F_k^{\prime}-(t_n+q_n)\|_{L^2(G)}\le
c_{1}\frac{1}{n^{\widetilde{s}}}+c_{2}\frac{1}{R^n},\quad n\in\mathbb{N}.
\end{equation}
Note that $\widetilde{s}=s^\star$, if $\gamma_{p_k+1}^{(k)}\notin\mathbb{N}$ for the index $k$ for which the minimum
is attained in the definition of $\widetilde{s}$. In the opposite case, where for the same index $k$,
it holds $\gamma_{p_k+1}^{(k)}\in\mathbb{N}$, we get $\widetilde{s}=s^\star$ in (\ref{eq:f-pn}) by simply subtracting
from $g(z)$ and adding to $h(z)$, in the right hand side of (\ref{eq:sum}), the derivative of the Cauchy integral on $l_N$, with density function $a_{p_k+1}^{(k)}(z-\tau_k)^{\gamma_{p_k+1}^{(k)}}$.
This observation and (\ref{eq:5}) imply that there exists a sequence of augmented polynomials
$\{\widetilde{p}_n\}$, with $\widetilde{p}_n\in\mathbb{P}^{A_2}_n$, such that,
\begin{equation}
\|K(\cdot,z_{0})-\widetilde{p}_n\|_{L^{2}(G)}\leq
c_{3}\frac{1}{n^{s^{\star}}}+c_{4}\frac{1}{R^{n}},\quad n\in\mathbb{N},
\end{equation}
and the rest goes in similar lines as in the proof of Theorem~\ref{th:1}, except here we use the version of
Andrievskii's lemma for functions with anti-derivatives in the space $\mathbb{P}^{A_2}_n$, given in \cite[Corollary~2.5]{MSS}.
These yield,
\begin{equation}
\|f_{0}-\widetilde{\pi}_{n}\|_{L^{\infty}(\overline{G})}\leqslant
c_{5}\sqrt{\log n}\frac{1}{n^{s^{\star}}}+c_{6}\sqrt{\log n}\frac{1}{R^{n}},\quad n\ge2,
\end{equation}
and (\ref{eq:27}) follows with a different, but still less than $\varrho$, $R$.
\end{proof}

\begin{remark}\label{rem:bmab}
Note that  $n^{s^\star}\le R^n$, as $n\to\infty$. Therefore from (\ref{eq:27}) we recover the result of \cite[Thm.~3.1]{MSS}. However, Theorem~\ref{th:5} above gives, in addition, the precise dependence of the BKM/AB error on the pole singularities of $f_0$  for small values of $n$. We also note the lower estimate
\begin{equation*}
\|f_{0}-\widetilde{\pi}_{n}\|_{L^{\infty}(\overline{G})}\ge
c\frac{1}{n^{s^{\star}}},\quad n\in\mathbb{N},
\end{equation*}
established in \cite[Thm.~3.2]{MSS}.
\end{remark}

\subsection{BKM/AB with pole and corner singularities}
We consider now the application of the BKM/AB with both pole and corner singular basis function of the form
studied in Sections \ref{ss:pole} and \ref{ss:corner}. Regarding poles we recall, in particular, our assumptions in Section~\ref{ss:pole}. That is, the nearest singularities of $f_{0}$ in $\Omega$ are $\kappa$ poles, each one of order $k_{j}$ at $z_{j}$, $j=1,2,\ldots \kappa$, where $|\Phi(z_{1})|\le|\Phi(z_{2})|\le\cdots\le|\Phi(z_{\kappa})|$, while the other singularities of $f_{0}$ occur at points $z_{\kappa+1},z_{\kappa+2},\ldots$, where
$|\Phi(z_{\kappa})|<|\Phi(z_{\kappa+1})|\leqslant|\Phi(z_{\kappa+2})|\leqslant\cdots$.
Therefore, for the BKM/AB we consider the system $\{\eta_j\}$, defined  by:
\begin{itemize}
\item[(i)]
the $\kappa$ pole singular functions (\ref{eq:7}), with $m_j=1$, $j=1,2,\ldots,\kappa$;
\item[(ii)]
the $r_M=\sum_{k=1}^Mp_k$ corner singular functions of the form (\ref{eq:8});
\item[(iii)]
and the $n$ monomials (\ref{eq:9}).
\end{itemize}
Accordingly, we let $\mathbb{P}_n^{A_{3}}$ denote the space,
$$
\mathbb{P}_{n}^{A_{3}}:=\{p:p(z)=
\sum_{j=1}^{\kappa+r_M+n}t_{j}\eta_{j}(z),\,\, t_{j}\in\mathbb{C}\},
$$
and note that the associated augmented polynomial $\widetilde{K}_{n}(z,z_{0})$ is
the best approximation to $K(z,z_{0})$ in $L^{2}(G)$ out of the space $\mathbb{P}^{A_{3}}_{n}$.

The following result is a version of Andrievsii's lemma for functions with anti-derivatives in   $\mathbb{P}_n^{A_{3}}$. It will be used below, in the proof of the concluding theorem of this section (in the transition from the $L^2(G)$-norm to the $L^\infty(\overline{G})$-norm, where we establish the BKM/AB error in approximating $f_0$ by the augmented polynomials $\widetilde{\pi}_n$ derived from $\mathbb{P}_n^{A_{3}}$.

\begin{lemma}\label{lem:andrii-3}
Assume that $\Gamma$ is piecewise analytic without cusps and let $t_k\in\Gamma$, $k=1,2,\ldots,m$. Also, let
$P\in\mathbb{P}_n$ and $q$ be a fixed polynomial with no zeros on $\overline{G}$. Assume further that
for some constants $a_{n,k,j}$, $k=1,2,\ldots,m$, $j=1,2,\ldots,r_k$, the function
$$
L(z):=\frac{P(z)}{q(z)}+\sum_{k=1}^m\sum_{j=1}^{r_k}a_{n,k,j}f_{\beta_j^{(k)},t_k}(z),
$$
where $f_{\beta_{j}^{(k)},t_k}(z)=(z-t_k)^{\beta_{j}^{(k)}}$, with  $\beta_{j}^{(k)}>0$ non-integer, satisfies: $L(z_{0})=0$ and $\|L^{\prime}\|_{L^{2}(G)}\leq M$. Then,
\begin{equation}\label{eq:36}
\|L\|_{L^{\infty}(\overline{G})}\leq CM\sqrt{\log n},
\end{equation}
where $C$ is a constant independent of $n$ and of $\{\{a_{n,k,j}\}_{j=1}^{r_k}\}_{k=1}^m$.
\end{lemma}
\begin{proof}
The proof is based on Andrievskii's lemma for singular algebraic functions given in
\cite[Corollary~2.5]{MSS} and relies on the results contained in \cite[\S2]{MSS}.
The details of the derivation are as follows:

First, we note that our assumption implies that $\Gamma$ is a quasiconformal curve. Then, it is straightforward to verify that the results of Theorems~2.1 and 2.2 (and hence the result of Corollary~2.2) in \cite{MSS} hold true for functions of the form $q^2(z)f_{\beta,\tau}(z)$, where $f_{\beta,\tau}(z):=(z-\tau)^\beta$, with $\tau\in\Gamma$ and $\beta>0$ non-integer. That is,
\begin{equation}\label{eq:Cor22ext}
\inf_{p\in\mathbb{P}_{n}}\|q^2f_{\beta,\tau}-p\|_{L^2({G})}\asymp\frac{1}{n^{(2-a)(\beta+1)}},
\end{equation}
where $\alpha\pi$ ($0<\alpha<2$) denotes the interior angle of $\Gamma$ at $\tau$.

With (\ref{eq:Cor22ext}) at hand it is, again, straightforward to verify consequentially that the results of
Theorem~2.3, Corollaries~2.3 and 2.4, Lemma~2.3 and Corollary~2.5, of \cite{MSS}, hold true if we replace $f_{\beta,\tau}^{\prime}$ by $q^{2}f_{\beta,\tau}^{\prime}$. In particular,  Corollary~2.5 of \cite{MSS} applied to the function
$$
S(z):=\int_{z_0}^zq^2(z)L^\prime(z)dz,
$$
where the path of integration $[z_0,z]$ is any rectifiable arc in $G$, gives that
\begin{equation*}
\|S\|_{L^{\infty}(\overline{G})}\leq c_1\sqrt{\log n}\,\|S^\prime\|_{L^2(G)}.
\end{equation*}
(Note that $S(z_0)=0$ and $S^\prime(z)=q^{2}(z)L^\prime(z)$.) Therefore, our hypothesis on $\|L^\prime\|_{L^2(G)}$  yields the inequality
\begin{equation}\label{eq:S-andri}
\|S\|_{L^{\infty}(\overline{G})}\leq c_2\sqrt{\log n}\,M.
\end{equation}
On the other hand we have,
$$
L(z)=\int_{z_0}^zq^{-2}(z)S^\prime(z)dz=
q^{-2}(z)S(z)+2\int_{z_0}^zq^{-3}(z)q^\prime(z)S(z)dz,
$$
which implies
$$
\|L\|_{L^{\infty}(\overline{G})}\leq c_3\,\|S\|_{L^{\infty}(\overline{G})}
$$
and (\ref{eq:36}) follows from (\ref{eq:S-andri}); cf.\ \cite[p.~122]{Ga87}.
\end{proof}

The concluding result of this section provides the theoretical justification for the use of the BKM/AB, with both corner and pole singularities.

Let $\widetilde{\pi}_n$ denote the BKM/AB approximation to $f_0$ resulting from the space  $\mathbb{P}_n^{A_3}$. Then we have the following:
\begin{theorem}\label{th:c}
Assume that $\Gamma$ is piecewise analytic without cusps and
set $\varrho:=|\Phi(z_{\kappa+1})|$ and  $s^{\star}:=\min\{(2-\alpha_k)\gamma_{\nu_k}^{(k)}:\ 1\leq k\le M\}$. Then,
\begin{equation}\label{eq:thc}
\|f_{0}-\widetilde{\pi}_{n}\|_{L^{\infty}(\overline{G})}\leq
c_{1}\sqrt{\log n}\frac{1}{n^{s^{\star}}}+c_{2}\frac{1}{R^{n}},\quad n\ge2,
\end{equation}
for any $R$, $1<R<\varrho$.
\end{theorem}
\begin{proof}
As in the proof of Theorem~\ref{th2}, we set $Q(z):=\prod_{j=1}^{\kappa}(z-z_{j})^{k_{j}+1}$. The result (\ref{eq:thc}) will emerge by working as in the proof of Theorem~\ref{th:5}. The basic idea is to consider, in a bigger domain $\widetilde{G}$, the anti-derivatives $F$ and $G_k$ of the functions $Qf_0^\prime$ and $QF_k^\prime$, respectively,
in the place of the functions $f_0$ and $F_k$. The details of the derivation are as follows:

We note that the function $Qf_0^\prime$ shares the same analytic properties with $f_0^\prime$, apart from the fact that it has the singularities at the points $z_j$,
$j=1,\ldots,\kappa$, all removed.
Therefore, the function
\begin{equation}\label{eq:Fdef}
F(z):=\int_{z_{0}}^{z}Q(\zeta)f_{0}^{\prime}(\zeta)d\zeta,
\end{equation}
can be extended analytically to a larger domain $\widetilde{G}$ than the one considered in Section~\ref{ss:corner}.
This larger domain $\widetilde{G}$ is obtained by choosing the point $\zeta_1$ close to the nearest pole $z_{\kappa+1}$ of $Qf_0^\prime$ in $\Omega$, but inside the level curve $L_\varrho$, where now
$\varrho:=|\Phi(z_{\kappa+1})|$. The remaining part of the construction of $\widetilde{G}$ is exactly the same as
in Section~\ref{ss:corner}.

It follows therefore that (\ref{eq:Fdef}) is valid for $z\in\widetilde{G}$, provided the arc of integration $[z_0,z]$ lies on
$\widetilde{G}\cup\partial\widetilde{G}\setminus\{\cup_{k=1}^M\tau_k\}$ and is rectifiable. (This is always possible because $\partial\widetilde{G}$ is piecewise analytic.)
Since the derivative of $f_0$ near $\tau_k$ can be obtained by termwise differentiation of the expansion (\ref{eq:lehmanO}), (cf.\ \cite[p.~1448]{Lehman}) and since any power in the resulting expansion is bigger than $-\frac{1}{2}$,
we see that $f_0^\prime$ is integrable along any
rectifiable arc in $\widetilde{G}$ with one endpoint at $\tau_k$. Therefore, integration by parts gives, for
$z\in\widetilde{G}\cup\partial\widetilde{G}$,
\begin{equation}\label{eq:Feq}
F(z)=Q(z)f_{0}(z)-\int_{z_{0}}^{\tau_{k}}Q^{\prime}(\zeta)f_{0}(\zeta)d\zeta
-\int_{\tau_{k}}^{z}Q^{\prime}(\zeta)f_{0}(\zeta)d\zeta,
\end{equation}
where we made use of the normalization of $f_0$ at $z_0$.
This shows that $F$ is continuous on $\partial\widetilde{G}$ and analytic and on
$\partial\widetilde{G}$, except for the endpoints $\tau_k$.
By arguing as in (\ref{eq:f0p}) we have, for $z\in G$,
\begin{equation}\label{eq:Fp}
Q(z)f_0^\prime(z)=F^{\prime}(z)=\frac{1}{2\pi i}\sum_{k=1}^{M}\int_{l_k}\frac{F(t)}{(t-z)^{2}}dt
+\frac{1}{2\pi i}\int_{l_N}\frac{F(t)}{(t-z)^{2}}dt.
\end{equation}

Similar properties to those of $F$ apply to the anti-derivative
\begin{equation}\label{eq:Gkdef}
G_k(z):=\int_{z_{0}}^{z}Q(\zeta)F_k^{\prime}(\zeta)d\zeta,
\end{equation}
of $QF_k^{\prime}$, $k=1,\ldots,M$. That is, for $z\in\widetilde{G}\cup\partial\widetilde{G}$,
\begin{equation}\label{eq:Gkeq}
\begin{alignedat}{2}
G_k(z)=Q(z)F_k(z)&-Q(z_0)F_k(z_0)\\
&-\int_{z_{0}}^{\tau_{k}}Q^{\prime}(\zeta)F_k(\zeta)d\zeta
-\int_{\tau_{k}}^{z}Q^{\prime}(\zeta)F_k(\zeta)d\zeta,
\end{alignedat}
\end{equation}
and, for $z\in G$,
\begin{equation}\label{eq:QFp}
Q(z)F_k^\prime(z)=G_k^{\prime}(z)=\frac{1}{2\pi i}\sum_{r=1}^{M}\int_{l_r}\frac{G_k(t)}{(t-z)^{2}}dt
+\frac{1}{2\pi i}\int_{l_N}\frac{G_k(t)}{(t-z)^{2}}dt.
\end{equation}

Next, by combining (\ref{eq:Feq}) and (\ref{eq:Gkeq}) we get,
\begin{equation}\label{eq:F-Gk-dk}
F(z)-d_k-G_{k}(z)=Q(z)[f_0(z)-F_k(z)]-
\int_{\tau_k}^zQ^\prime(\zeta)[f_0(z)-F_k(\zeta)]d\zeta,
\end{equation}
where
\begin{equation*}
d_k:=Q(z_0)F_k(z_0)-\int_{z_0}^{\tau_k}Q^\prime(\zeta)[f_0(\zeta)-F_k(\zeta)]d\zeta.
\end{equation*}
This and (\ref{eq:lehmanO}) lead to
\begin{equation}\label{eq:FdkGkatlk}
|F(z)-d_k-G_{k}(z)|\preceq|z-\tau_{k}|^{\gamma_{p_k+1}^{(k)}},\quad z\in l_k.
\end{equation}

By reasoning as in the proof of Theorem \ref{th:5} we conclude, by using (\ref{eq:Fp}) and (\ref{eq:QFp}) that, for $z\in G$,
\begin{equation}\label{eq:Qfsum}
Q(z)f_{0}^{\prime}(z)-Q(z)\sum_{k=1}^{M}F_k^{\prime}(z)=g(z)+h(z),
\end{equation}
where the singular part
\begin{equation}\label{eq:gFdef}
g(z):=\frac{1}{2\pi i}\sum_{k=1}^{M}\int_{l_{k}}\frac{F(t)-d_k-G_k(t)}{(t-z)^{2}}dt,
\end{equation}
of the splitting (\ref{eq:Qfsum}) can be approximated, eventually,  by a sequence of polynomials $\{q_{n}\}_{n=1}^{\infty}$
at a polynomial rate, viz.,
\begin{equation}
\|g-q_{n}\|_{L^{2}(G)}\preceq\frac{1}{n^{s^\star}},\quad n\in\mathbb{N},
\end{equation}
with $s^{\star}:=\min\{(2-\alpha_k)\gamma_{\nu_k}^{(k)}:\ 1\leq k\le M\}$
and the analytic part
\begin{equation}\label{eq:hFonlr}
\begin{alignedat}{2}
h(z)
&:=\frac{1}{2\pi i}\sum_{k=1}^M\int_{l_k}\frac{d_k}{(t-z)^{2}}dt+
\frac{1}{2\pi i}\int_{l_N}\frac{F(t)}{(t-z)^{2}}dt\\
&-\frac{1}{2\pi i}
\sum_{k=1}^M\sum_{r=1\atop{r\neq k}}^{M}\int_{l_r}\frac{G_{k}(t)}{(t-z)^{2}}dt
-\frac{1}{2\pi i}\sum_{k=1}^M\int_{l_N}\frac{G_{k}(t)}{(t-z)^{2}}dt.
\end{alignedat}
\end{equation}
can be approximated by a sequence of polynomials $\{t_{n}\}_{n=1}^{\infty}$
at a geometric rate, viz.,
\begin{equation}
\|h-t_n\|_{L^{\infty}(\overline{G})}\preceq\frac{1}{R^{n}},\quad n\in\mathbb{N},
\end{equation}
where $1<R<\varrho$. Hence using the triangle inequality we get
\begin{equation}%\label{eq:Qf-pn}
\|Q(f_{0}^{\prime}-\sum_{k=1}^{M}F_k^{\prime})-(t_n+q_n)\|_{L^2(G)}\le
c_{1}\frac{1}{n^{s^\star}}+c_{2}\frac{1}{R^n}.
\end{equation}
This implies
\begin{equation}%\label{eq:Qf-pn}
\left\|f_{0}^{\prime}-\sum_{k=1}^{M}F_k^{\prime}-\frac{t_n+q_n}{Q}\right\|_{L^2(G)}\le
\frac{c}{d^{\kappa+\xi}}\left[c_{1}\frac{1}{n^{s^\star}}+c_{2}\frac{1}{R^n}\right],
\end{equation}
where $d:=\min_{j=1,2,\ldots,\kappa}\{|z-z_{j}|:\,z\in \Gamma\}$ and
$\xi:=\sum_{j=1}^{\kappa}k_j$. Thus, from (\ref{eq:5}) we conclude
there exists a sequence of augmented polynomials
$\{\widetilde{p}_n\}$, where $\widetilde{p}_n\in\mathbb{P}^{A_3}_n$, such that,
\begin{equation}
\|K(\cdot,z_{0})-\widetilde{p}_n\|_{L^{2}(G)}\leq
c_{3}\frac{1}{n^{s^{\star}}}+c_{4}\frac{1}{R^{n}},\quad n\in\mathbb{N},
\end{equation}
Therefore, using the minimum property of the augmented kernel polynomials,  we have
\begin{equation}
\|K(\cdot,z_0)-\widetilde{K}_n(\cdot,z_0)\|_{L^{2}(G)}\leq c_{3}\frac{1}{n^{s^{\star}}}+c_{4}\frac{1}{R^{n}},\quad n\in\mathbb{N},
\end{equation}
and this, in conjunction with the equivalence Lemma \ref{l:1}, yields that
\begin{equation}
\|f_{0}^{\prime}-\widetilde{\pi}_{n}^{\prime}\|_{L^{2}(G)}\leq c_{5}\frac{1}{n^{s^{\star}}}+c_{6}\frac{1}{R^{n}},\quad n\in\mathbb{N}.
\end{equation}

Since,
\begin{equation}
\widetilde{\pi}_n(z)=\frac{P(z)}{q(z)}+\sum_{k=1}^M\sum_{j=0}^{p_k}a_{n,k,j}(z-\tau_k)^{\gamma_j^{(k)}},
\end{equation}
where $q(z):=\prod_{j=1}^\kappa(z-z_j)^{k_j}$ and $P(z)$ is a polynomial of degree $n+\xi$, the rest goes
as the concluding part of the proof of Theorem~\ref{th:5}, except here we use the result of Lemma~\ref{lem:andrii-3} in the place of \cite[Corollary~2.5]{MSS}.
\end{proof}

\section{numerical results}\label{sec:5}

In this section we present numerical examples, that illustrate the convergence results predicted by the theory of Sections \ref{sec:3} and \ref{sec:4}, regarding the following four errors:
 \begin{equation}\label{eq:37}
 E_{n,2}(K,G):=\|K(\cdot,z_{0})-K_{n}(\cdot,z_{0})\|_{L^{2}(G)},
\end{equation}
\begin{equation}\label{eq:38}
E_{n,\infty}(f_0,G):=\|f_{0}-\pi_{n}\|_{L^{\infty}(\overline{G})},
\end{equation}
\begin{equation}\label{eq:39}
 \widetilde{E}_{n,2}(K,G):=\|K(\cdot,z_{0})-\widetilde{K}_{n}(\cdot,z_{0})\|_{L^{2}(G)},
 \end{equation}
\begin{equation}\label{eq:40}
 \widetilde{E}_{n,\infty}(f_{0},G):=\|f_{0}-\widetilde{\pi}_{n}\|_{L^{\infty}(\overline{G})}.
\end{equation}

We do this by considering two different geometries: (a) lens-shaped domains; and (b) circular sectors. In both cases the normalized conformal map $f_0$, and hence the kernel function $K(\cdot,z_0)$, are known explicitly in terms of elementary functions. In addition, we present results illustrating the decay of the two sequences of points $\{P_n(z_0)\}_{n=1}^\infty$ and $\widetilde{P}_n(z_0)_{n=1}^\infty$ of the Bergman polynomials.

\subsection{Computational details}\label{sec:5.1}
Let $\{\eta_j\}$ denote the set of linearly independent functions defined in (\ref{eq:7})--(\ref{eq:9}).
For the application of the BKM/AB (or BKM), we compute the associated orthonormal set $\{\widetilde{P}_j\}$ by using the Arnoldi variant of the Gram-Schmidt (GS) process studied in \cite{NS-AGS}, rather than the conventional GS, which is based on the orthonormalization of the monomials $\{z^j\}$, as it is suggested in \cite{LPS} and \cite{PW86}.
In the Arnoldi GS we construct first the polynomial part of the set $\{\widetilde{P}_j\}$ by orthonormalizing consequently the functions $1,z\widetilde{P}_0,z\widetilde{P}_1,\ldots,z\widetilde{P}_{n-1}$. Then, we orthonormalize the singular basis functions (\ref{eq:7}) and (\ref{eq:8}). As it is shown in \cite{NS-AGS},
in this way we avoid the instability difficulties associated with the application of the conventional GS method.
For a comprehensive report of experiments testifying the instability of the conventional GS in BKM and BKM/AB we
refer to \cite[\S5]{PW86}.

The GS process, requires the computation of inner products of the form
\begin{equation}\label{eq:inner}
\langle \eta_k,\eta_l\rangle=\int_{G}\eta_k(z)\,\overline{\eta_l(z)}\,dA(z).
\end{equation}
For our purposes here, we compute these inner products by using Green's formula in order to transform the area integral into a line integral. For instance, when $\eta_k=z^{k}$, $\eta_l=z^{l}$, we have
\begin{equation}
\langle z^{k},z^{l}\rangle=\frac{1}{2(l+1)i}\int_{\Gamma}z^{k}\,\overline{z}^{l+1}\,dz.
\end{equation}
In all cases considered below this leads to explicit formulas for the inner products (\ref{eq:inner}).

Regarding the computation of the errors (\ref{eq:37})--(\ref{eq:40}) we note the following:
\begin{itemize}
\item[(i)]
 The two errors $\|K(\cdot,z_0)-K_{n}(\cdot,z_0)\|_{L^{2}(G)}$
 and $\|K(\cdot,z_0)-\widetilde{K}_{n}(\cdot,z_0)\|_{L^{2}(G)}$
 are computed by using Parseval's
identity, i.e.,
\begin{equation}
\|K(\cdot,z_0)-K_{n}(\cdot,z_0)\|_{L^{2}(G)}^{2}=K(z_0,z_0)-K_{n}(z_0,z_0),
\end{equation}
and
\begin{equation}
\|K(\cdot,z_0)-\widetilde{K}_{n}(\cdot,z_0)\|_{L^{2}(G)}^{2}=K(z_0,z_0)-\widetilde{K}_{n}(z_0,z_0).
\end{equation}
\item[(ii)]
Estimates for the two errors $ \|f_{0}-\pi_{n}\|_{L^{\infty}(\overline{G})}$
and
$\|f_{0}-\widetilde{\pi}_{n}\|_{L^{\infty}(\overline{G})}$
are obtained by using the exact formula for $f_0$ and then sampling the differences $f_0-\pi_{n}$ and $f_0-\widetilde{\pi}_{n}$ on $100$ uniformly distributed points on each analytic arc forming the boundary $\Gamma$.
\end{itemize}

All results were obtained with Maple $11$, using the systems facility for $64$-digit floating point arithmetic, on a pentium PC.

\subsection{BKM and BKM/AB approximation}\label{sec:5.2}
\subsubsection{Lens-shaped domains}
Let $G_{a,b}$ denote the lens-shaped domain, whose
boundary $\Gamma$ consists of two circular arcs $\Gamma_a$ and
$\Gamma_b$ that join together the points $i$ and $-i$ ($\Gamma_a$ being to the
left of $\Gamma_b$) and form angles $a$ and $b$, respectively,
with the linear segment $[-i,i]$.
(Thus, with the notation of Section~\ref{subsec:corners} we have $\alpha_1=\alpha_2=\alpha$,
where $\alpha:=(a+b)/\pi$.) Let $f_{0}$ denote the normalized conformal map from $G_{a,b}$ onto $D(0,r_{0})$, with $f_0(0)=0$ and $f^{\prime}_{0}(0)=1$.
By working as in \cite[\S4]{M-DSS}, it is easy to check that, if $a+b=k\pi/m$, where $k, m\in\mathbb{N}$, then $f_{0}$ is given by
\begin{equation}\label{eq:41}
f_{0}(z)=r_0\frac{\left[\frac{z-i}{z+i}\right]^{\frac{m}{k}}-(-1)^{\frac{m}{k}}}
{\left[\frac{z-i}{z+i}\right]^{\frac{m}{k}}-(-1)^{\frac{m}{k}}e^{-2ia\frac{m}{k}}},
\quad z\in\overline{G}_{a,b},
\end{equation}
where $r_0=(k/m)\sin(ma/k)$. Also,
\begin{equation}\label{eq:42}
K(z,0)=-\frac{4m^{2}}{\pi k^2}\frac{\left[(z-i)(z+i)\right]^{\frac{m}{k}-1}}
{\left[e^{ia\frac{m}{k}}(-i)^{\frac{m}{k}}(z-i)^{\frac{m}{k}}-e^{-ia\frac{m}{k}}
(i)^{\frac{m}{k}}(z+i)^{\frac{m}{k}}\right]^2},
\end{equation}
and thus
$$
K(0,0):=\frac{m^2}{\pi k^2}\frac{1}{\sin^{2}(ma/k)}.
$$

It is also easy to verify that the formulas (71)--(73) of \cite{M-DSS} work as well for the exterior conformal map
$\Phi:\overline{\mathbb{C}}\backslash \overline{G}_{a,b}\rightarrow\Delta$ consider here.
That is, $w=\Phi(z)$ is given by the composition of the following three transformations:
 \begin{equation}\label{eq:422}
 \xi(z):=e^{i((m-k)\pi/m+a)}\frac{z-i}{z+i},
 \end{equation}
  \begin{equation}\label{eq:423}
  t(\xi):=\xi^{m/(2m-k)},\quad \arg\xi\in(-k\pi/m,(2m-k)\pi/m],
   \end{equation}
 \begin{equation}\label{eq:424}
  w(t):=\frac{1-\lambda_{a}t}{t-\lambda_{a}}, \quad
  \lambda_{a}:=e^{i((m-k)\pi+ma)/(2m-k)}.
  \end{equation}

We consider separately the following three cases:
\begin{itemize}
\item[(i)]
$\alpha=1/2$, with $a=\pi/6$ and $b=\pi/3$;
\item[(ii)]
$\alpha=1/2$, with $a=\pi/4$ and $b=\pi/4$;
\item[(iii)]
$\alpha=2/13$, with $a=\pi/13$ and $b=\pi/13$.
\end{itemize}

\noindent
\textit{Cases (i) and (ii)}:
In the first two cases the conformal map $f_0$ is a rational function, and hence it has an analytic
continuation across $\Gamma$ into $\Omega$.
When $a=\pi/6$, then the two nearest singularities of $f_{0}$ in $\Omega$ are the two simple poles at $z_{1}=-\sqrt{3}/3$ and $z_{2}=\sqrt{3}$, where $|\Phi(z_{1})|\approx1.347$ and $|\Phi(z_{2})|\approx2.532$.
Accordingly, in our experiments, we use the singular function $[1/(z-z_1)]^{\prime}$.
This cancels out the nearest singularity at $z_1$.  In the symmetric case, where $a=b=\pi/4$, we have
$$
f_0(z)=\frac{-2iz}{z^2-1},
$$
and the only singularities of $f_0$ are the two simple poles at $z_{1}=-1$ and $z_{2}=1$,
where $|\Phi(z_{1})|=|\Phi(z_{2})|=\sqrt{3}$.
In this case, we use the singular function $[z/(z^2-z_1^2)]^{\prime}$, which takes care of both poles at $z_1$
and $z_2$.
% cf. \cite[p.~656]{PW86}.
It follows from Remark \ref{rem:PapWa} that this cancels out all the singularities of $f_0$.

We recall from Theorems \ref{th:1} and \ref{th2ii} (and their proof) the four estimates,
\begin{equation}\label{e:1}
E_{n,2}(K,G)\preceq\frac{n}{|\Phi(z_{1})|^{n}},
\end{equation}
\begin{equation}\label{e:2}
\frac{1}{|\Phi(z_{1})|^{n}}\preceq E_{n,\infty}(f_0,G)\preceq\frac{n\sqrt{\log n}}{|\Phi(z_{1})|^{n}},
\end{equation}
and
\begin{equation}\label{e:3}
 \widetilde{E}_{n,2}(K,G)\preceq\frac{n}{|\Phi(z_{2})|^{n}},
 \end{equation}
\begin{equation}\label{e:4}
\frac{1}{|\Phi(z_{2})|^{n}}\preceq \widetilde{E}_{n,\infty}(f_0,G)\preceq\frac{n\sqrt{\log n}}{|\Phi(z_{2})|^{n}}.
\end{equation}
Below, we present numerical results that illustrate the laws of the above errors and rates. In presenting the numerical results we use the following notation:

\begin{itemize}
\item
$\varrho:$ This denotes the order of approximation (the base of $n$) in the errors $(\ref{e:1})$--$(\ref{e:4})$.
\item
$\varrho_n:$ This denotes the estimate of $\varrho$, corresponding to $n$, and is determined as follows: With $E_{n}$  denoting any of the two errors $E_{n,2}(K,G)$ or $\widetilde{E}_{n,2}(K,G)$, we assume that
\begin{equation}\label{eq:En-rho}
E_{n}\approx c\frac{n}{\varrho^{n}}
\end{equation}
and seek to estimate $\varrho$ by means of the formula,
\begin{equation}\label{e:5}
\varrho_{n}=\left(\frac{n}{n-m}\frac{E_{n-m}}{E_{n}}\right)^{\frac{1}{m}}.
\end{equation}
(Here we take $m=4$, or $m=5$.)
If $E_{n}$  denotes either of the two errors $E_{n,\infty}(f_0,G)$ or $\widetilde{E}_{n,\infty}(f_{0},G)$, then
we assume that
\begin{equation}
E_{n}\approx c\frac{n\sqrt{\log n}}{\varrho^{n}},
\end{equation}
and seek to estimate $\varrho$ by means of the formula,
\begin{equation}\label{e:6}
\varrho_{n}=\left(\frac{n}{n-m}\frac{\sqrt{\log n}}{\sqrt{\log(n-m)}}\frac{E_{n-m}}{E_{n}}\right)^{\frac{1}{m}},
\end{equation}
with $m=4$, or $m=5$.
\item
$\varrho_n^{\star}:$ With $E_{n}$  denoting either of the errors $E_{n,\infty}(f_0,G)$ or $\widetilde{E}_{n,\infty}(f_{0},G)$,
we also test the law
\begin{equation}\label{e:8}
E_{n}\approx c\frac{1}{\varrho^{n}},
\end{equation}
thereby estimating $\varrho$ by means of
\begin{equation}\label{e:7}
\varrho_{n}^{\star}=\left(\frac{E_{n-m}}{E_{n}}\right)^{\frac{1}{m}}.
\end{equation}
%by using the same values of $m$ as in (\ref{e:6}).
\end{itemize}

The presented results show clearly the advantage of the BKM/AB over the BKM. In addition, they
indicate a close agreement between the theoretical and the computed order of approximation.
In Tables~$\ref{tab:5.1}$ and $\ref{tab:5.3}$, the results associated with the errors $E_{n,2}(K,G)$ and $\widetilde{E}_{n,2}(K,G)$ indicate the convergence of $\varrho_{n}$ to $\varrho$. Regarding the errors $E_{n,\infty}(f_0,G)$ and $\widetilde{E}_{n,\infty}(f_{0},G)$, the results of the Tables~$\ref{tab:5.2}$ and $\ref{tab:5.4} $ show that $\varrho_{n}^{\star}$ converges faster to $\varrho$ than $\varrho_{n}$.
This suggest, at least for the geometry under consideration, a behavior of the type (\ref{e:8}) for the errors $E_{n,\infty}(f_0,G)$ and $\widetilde{E}_{n,\infty}(f_{0},G)$.
As it is predicted by Remark \ref{rem:PapWa}, in Case (ii) the two errors $\widetilde{E}_{n,2}(K,G)$ and $\widetilde{E}_{n,\infty}(f_{0},G)$ vanish. This was testified in our experiments, in the sense that the computed errors $\widetilde{E}_{n,2}(K,G)$ and $\widetilde{E}_{n,\infty}(f_0,G)$ were zero within machine precision, thus they are not quoted in Tables $\ref{tab:5.3}$ and $\ref{tab:5.4}$.

\begin{table}[t]
\begin{center}
\begin{tabular}{|r|c|c|c|c|c|c|c|c|}
\hline
${\ } $
& \multicolumn{2}{|c}{\textrm{BKM:}$\,\, \varrho\approx 1.347$}
& \multicolumn{2}{c|}{\textrm{BKM/AB:}$\,\, \varrho\approx 2.532$}
${\vphantom{{\sum^{\sum^N}}}}$ \\*[3pt]
\cline{2-5}
    $n$ & $E_{n,2}(K,G)$ & $\varrho_{n}$ & $\widetilde{E}_{n,2}(K,G)$ & $\varrho_{n}$
    ${\vphantom{\sum^{\sum^{\sum^N}}}}$ \\*[3pt]    \hline
    5  & 4.4e-01 &    -   &2.7e-02&   -  \\
    10 & 1.3e-01 & 1.47 &3.6e-04&2.72\\
    15 & 3.5e-02 & 1.41 &4.1e-06&2.65\\
    20 & 8.9e-03 & 1.39 &4.6e-08&2.60\\
    25 & 2.2e-03 & 1.38 &4.9e-10&2.59\\
    30 & 5.4e-04 & 1.37 &5.2e-12&2.57\\
    35 & 1.3e-04 & 1.36 &5.4e-14&2.57\\ \hline
\end{tabular}
\end{center}

\medskip
\caption{BKM approximations to $K$: Lens-shaped, Case~(i).}
\label{tab:5.1}
\end{table}

\begin{table}[t]
\begin{center}
\begin{tabular}{|c|c|c|c|c|c|c|c|c|}
\hline
${\ } $
&\multicolumn{3}{|c}{{BKM:}$\,\, \varrho\approx 1.347$}
&\multicolumn{3}{c|}{{BKM/AB:}$\,\, \varrho\approx 2.532$}
${\vphantom{{\sum^{\sum^N}}}}$ \\*[3pt]
\cline{2-7}
$n$ & $E_{n,\infty}(f_{0},G)$ & $\varrho_{n}^{\star}$ & $\varrho_{n}$&
$\widetilde{E}_{n,\infty}(f_{0},G)$
 & $\varrho_{n}^{\star}$&$\varrho_{n}$
 ${\vphantom{\sum^{\sum^{\sum^N}}}}$ \\*[3pt]  \hline
 5  &2.5e-01&   - &  - &1.4e-02&   - &  - \\
 10 &6.8e-02&1.299&1.54&1.3e-04&2.541&3.03\\
 15 &1.6e-02&1.331&1.47&1.3e-06&2.528&2.79\\
 20 &3.8e-03&1.342&1.43&1.2e-08&2.537&2.71\\
 25 &8.5e-04&1.346&1.42&1.1e-10&2.532&2.67\\
 30 &1.9e-04&1.347&1.41&1.1e-12&2.532&2.64\\
 35 &4.3e-05&1.347&1.40&1.1e-14&2.532&2.62\\ \hline
\end{tabular}
\end{center}

\medskip
\caption{BKM approximations to $f_{0}$: Lens-shaped, Case~(i).}
\label{tab:5.2}
\end{table}

\begin{table}[t]
\begin{center}
\begin{tabular}{|c|c|c|c|c|c|c|c|c|}
\hline
${\ } $
&\multicolumn{1}{|c}{\textrm{BKM:}$\,\, \varrho\approx 1.732$}&
${\vphantom{{\sum^{\sum^N}}}}$ \\*[3pt]
\cline{1-3}
$n$ & $E_{n,2}(K,G)$ & $\varrho_{n}$
${\vphantom{\sum^{\sum^{\sum^N}}}}$ \\*[3pt] \hline
   4  & 2.7e-01 &    - \\
   8  & 4.0e-02 & 1.92 \\
   12 & 5.3e-03 & 1.83 \\
   16 & 6.7e-04 & 1.80 \\
   20 & 8.3e-05 & 1.78 \\
   24 & 1.0e-05 & 1.78 \\
   28 & 1.2e-06 & 1.77 \\
   32 & 1.4e-07 & 1.76 \\
   36 & 1.7e-08 & 1.75 \\ \hline
\end{tabular}
\end{center}
\medskip\caption{BKM approximations to $K$: Lens-shaped, Case~(ii).}
\label{tab:5.3}
\end{table}

\begin{table}[t]
\begin{center}
\begin{tabular}{|c|c|c|c|c|c|c|c|c|}
\hline
${\ } $
&\multicolumn{2}{|c}{\textrm{BKM:}$\,\, \varrho\approx1.732$}&
${\vphantom{{\sum^{\sum^N}}}}$ \\*[3pt]
\cline{2-4}
$n$ & $E_{n,\infty}(f_{0},G)$ & $\varrho_{n}^{\star}$ & $\varrho_{n}$
${\vphantom{\sum^{\sum^{\sum^N}}}}$ \\*[3pt] \hline
4  &1.3e-01&   - &  - \\
8  &1.6e-02&1.685&2.11\\
12 &1.8e-03&1.718&1.95\\
16 &2.0e-04&1.729&1.89\\
20 &2.3e-05&1.732&1.83\\
24 &2.5e-06&1.732&1.83\\
28 &2.8e-07&1.732&1.81\\
32 &3.1e-08&1.732&1.80\\
36 &3.4e-09&1.732&1.79\\ \hline
\end{tabular}
\end{center}

\medskip
\caption{BKM approximations to $f_{0}$: Lens-shaped, Case~(ii).}
\label{tab:5.4}
\end{table}

\medskip
\noindent
\textit{Case (iii)}:
In this case the conformal map $f_0$ has a branch point singularity at each of the two corners $\tau_1=i$ and $\tau_2=-i$, and therefore Lehman's expansions (\ref{eq:lemman2}) are valid with
$\gamma_1^{(1)}=\gamma_1^{(2)}=13/2$ and $\gamma_2^{(1)}=\gamma_2^{(2)}=1+1/\alpha=15/2$.
This gives $(2-\alpha)/\alpha=12$ and $(2-\alpha)(1+1/\alpha)=180/13=13.84\cdots$.
Furthermore, it follows from (\ref{eq:41}) that the nearest singularities of $f_0$ in $\Omega$, are the two simple poles at $z_{1}=\tan(\pi/13)$ and $z_{2}=-\tan(\pi/13)$, where
$|\Phi(z_{1})|=|\Phi(z_{2})|\approx1.119$, and the next singularity occurs at a point $z_{3}$, where $|\Phi(z_{3})|\approx2.055$.

Therefore, from Theorem~\ref{th3} we have that,
\begin{equation}\label{e:11}
E_{n,2}(K,G) \leqslant c_{1}\frac{1}{n^{12}}+c_{2}\frac{2}{R^{n}},
\end{equation}
 and
\begin{equation}\label{e:12}
 E_{n,\infty}(f_0,G)\leqslant c_{3}\sqrt{\log n}\frac{1}{n^{12}}+c_{4}\frac{1}{R^{n}},
\end{equation}
where $1<R<|\Phi(z_{1})|$.
In order to decide which singular functions to include in the BKM/AB the following estimates,
valid for $n=32$, are relevant; see also Theorem~\ref{th:c}:
\begin{equation*}
\begin{alignedat}{2}
\frac{1}{n^{(2-\alpha)/\alpha}}&\approx 8.7\times 10^{-19}, \quad
&\frac{1}{|\Phi(z_{1})|^{n}}\approx 2.7 \times 10^{-2}, \\
\frac{1}{n^{(2-\alpha)(1+1/\alpha)}}&\approx 1.4\times 10^{-21},
&\frac{1}{|\Phi(z_{2})|^{n}} \approx 1.0 \times 10^{-10}.
\end{alignedat}
\end{equation*}
The estimates in the first line indicate that for $n=32$ (even for bigger values of $n$), the dominant term in the errors $(\ref{e:11})$ and $(\ref{e:12})$ is $c_2\frac{1}{R^{n}}$.
As it is suggested by the estimate in the second line,  we use in our BKM/AB approximations only the singular function $[z/(z^2-z_1^2)]^{\prime}$, which takes care of the two symmetric poles at $z_1$ and $z_2$, and we
include no basis functions reflecting the corner singularities of $f_0$ on $\Gamma$.
Then, from Theorem~\ref{th:c} we have for the resulting approximations that
\begin{equation}\label{e:13}
\widetilde{E}_{n,2}(K,G)\leq c_{1}\frac{1}{n^{12}}+ c_{2}\frac{1}{R^{n}},
\end{equation}
\begin{equation}\label{e:14}
\widetilde{E}_{n,\infty}(f_0,G)\leqslant
c_{3}\sqrt{\log n}\frac{1}{n^{12}}+c_{4}\frac{1}{R^{n}},
\end{equation}
where $1<R<|\Phi(z_{2})|$.

Below, we present numerical results that illustrate the rates in (\ref{e:11})--(\ref{e:14}). In presenting the numerical results we use the following notation:
\begin{itemize}
\item
$\varrho:$ This denotes the order of approximation (the base of $n$) in the errors $(\ref{e:11})$--$(\ref{e:14})$.
\item
$\varrho_n:$ This denotes the estimate of $\varrho$, corresponding to $n$, and is determined as follows: With $E_{n}$  denoting any of the four errors $E_{n,2}(K,G)$, $\widetilde{E}_{n,2}(K,G)$, $E_{n,\infty}(f_0,G)$ or $\widetilde{E}_{n,\infty}(f_{0},G)$ we assume that
\begin{equation}\label{e:9}
E_{n}\approx c\frac{1}{\varrho^{n}},
\end{equation}
and seek to estimate $\varrho$ by means of the formula,
\begin{equation}\label{e:10}
\varrho_{n}=\left(\frac{E_{n-4}}{E_{n}}\right)^{\frac{1}{4}}.
\end{equation}
\end{itemize}

The results quoted in Tables~\ref{tab:5.7} and \ref{tab:5.8},
show the remarkable approximation achieved by the BKM/AB by using as little as 32 monomials. Moreover, they highlight the significance of Theorem~\ref{th:c}, as it is compared to the estimate (\ref{eq:6}),
in the sense that they confirm fully the theoretical prediction that the two poles at $z_1$ and $z_2$ are the most serious singularities of $f_0$  for small values of $n$; see also Remark~\ref{rem:bmab}.

\begin{table}[t]
\begin{center}
\begin{tabular}{|c|c|c|c|c|c|c|c|c|}
\hline
   ${\ } $
& \multicolumn{2}{|c}{\textrm{BKM:}$\,\, \varrho\approx1.119$}
& \multicolumn{2}{c|}{\textrm{BKM/AB:}$\,\, \varrho\approx2.055$}
${\vphantom{{\sum^{\sum^N}}}}$ \\*[3pt]
\cline{2-5}
$n$ & $E_{n,2}(K,G)$ & $\varrho_{n}$ &
$\widetilde{E}_{n,2}(K,G)$ &
$\varrho_{n}$
${\vphantom{\sum^{\sum^{\sum^N}}}}$ \\*[3pt] \hline
4  & 2.8819 &    -   &7.3e-02&  -  \\
8  & 2.3812 & 1.049 &5.6e-03&1.898\\
12 & 1.3864 & 1.145 &3.9e-04&1.934\\
16 & 0.9188 & 1.108 &2.6e-05&1.965\\
20 & 0.5961 & 1.114 &1.7e-06&1.974\\
24 & 0.3812 & 1.118 &1.1e-07&1.982\\
28 & 0.2413 & 1.121 &7.2e-09&1.981\\
32 & 0.1538 & 1.119 &4.6e-10&1.992\\ \hline
\end{tabular}
\end{center}

\medskip
\caption{BKM approximations to $K$: Lens-shaped, Case~(iii).}
\label{tab:5.7}
\end{table}

\begin{table}[t]
\begin{center}
\begin{tabular}{|c|c|c|c|c|c|c|c|c|}
\hline
${\ } $
& \multicolumn{2}{|c}{\textrm{BKM:}$\,\, \varrho\approx1.119$}
& \multicolumn{2}{c|}{\textrm{BKM/AB:}$\,\, \varrho\approx2.055$}
${\vphantom{{\sum^{\sum^N}}}}$ \\*[3pt]
\cline{2-5}
$n$ & $E_{n,\infty}(f_{0},G)$ & $\varrho_{n}$ &
$\widetilde{E}_{n,\infty}(f_{0},G)$ &
$\varrho_{n}$
${\vphantom{\sum^{\sum^{\sum^N}}}}$ \\*[3pt] \hline
4  & 0.8820 &    -   &1.2e-02&   -  \\
8  & 0.3817 & 1.232 &6.1e-04&2.101\\
12 & 0.2044 & 1.170 &3.4e-05&2.053\\
16 & 0.1180 & 1.147 &2.0e-06&2.029\\
20 & 0.0702 & 1.139 &1.2e-07&2.028\\
24 & 0.0424 & 1.135 &7.0e-09&2.028\\
28 & 0.0259 & 1.131 &4.1e-10&2.028\\
32 & 0.0160 & 1.128 &2.5e-11&2.028\\ \hline
\end{tabular}
\end{center}

\medskip
\caption{BKM approximations to $f_{0}$: Lens-shaped, Case~(iii).}
\label{tab:5.8}
\end{table}

\subsubsection{Circular sector}\label{subsec:5.2}
Let $G_\alpha$ denote the symmetric circular sector of radius $2$ and opening angle $\alpha\pi,\,\,0<\alpha<2$, at the origin, i.e.,
$$
G_\alpha:=\{z: |z|<2,\,\,-\alpha\pi/2<\arg z<\alpha\pi/2\}.
$$
Let $f_{0}$ denote the normalized conformal map from $G_\alpha$ onto $D(0,r_{0})$, with
$f_{0}(1)=0$ and $f^{\prime}_{0}(1)=1$.
For each value of the parameter $\alpha$ the conformal map $f_{0}(z)$ can be computed by means of the transformations (see \cite[p.~532]{MSS}):
\begin{equation}
 f_0(z)=\bigg[\frac{2\alpha(4^{1/\alpha}-1)}{4^{1/\alpha}+1}\bigg]\frac{t-d}{td-1},
\end{equation}
where
\begin{equation}
t=\bigg(\frac{iz^{1/\alpha}+2^{1/\alpha}}{iz^{1/\alpha}-2^{1/\alpha}}\bigg)^{2}\quad \textrm{and}\quad  d=\bigg(\frac{i+2^{1/\alpha}}{i-2^{1/\alpha}}\bigg)^2.
\end{equation}
This gives
\begin{equation}
r_0=\frac{2\alpha(4^{1/\alpha}-1)}{4^{1/\alpha}+1}\quad \textrm{and}\quad
K(1,1)=\frac{1}{\pi}\bigg(\frac{4^{1/\alpha}+1}{2\alpha(4^{1/\alpha}-1)}\bigg)^2.
\end{equation}

The normalized exterior map $\Phi:\mathbb{\overline{C}}\backslash \overline{G}_\alpha\rightarrow\Delta$
is given, as can be easily verified, by the composition of the
following three transformations:
\begin{equation}
\xi(z):=\frac{i(2^{1-1/\alpha}z^{1/\alpha}-2i)}{2^{1-1/\alpha}z^{1/\alpha}+2i},\quad\arg z\in(-\pi,\pi],
\end{equation}
\begin{equation}
t(\xi):=\xi^{2/3}, \quad\arg\xi\in(-\pi/2,3\pi/2],
\end{equation}
\begin{equation}
w(t):=\frac{1- e^{i\pi/3}t}{t-e^{i\pi/3}}.
\end{equation}

We consider separately the following two cases:
\begin{itemize}
\item[(i)]
$\alpha=1$\,\,\, (half-disk);
\item[(ii)]
$\alpha=3/2$\,\,\, (three-quarter disk).
\end{itemize}

\noindent
\textit{Case (i)}:
When $\alpha=1$, then the domain $G_\alpha$ is the half-disk
$$
G_1=\{z:|z|<2,\mathfrak{R}z>0\}.
$$
In this case the conformal map $f_0$ has an analytic continuation across $\Gamma$ into $\Omega$.
The nearest singularities of $f_{0}$ in $\Omega$, are the two simple poles at $z_{1}=-1$
and $z_{2}=4$, where $|\Phi(z_{1})|\approx 1.452$ and $|\Phi(z_{2})|\approx 2.212$.
Accordingly, in our experiments we use the singular function $[1/(z-z_1)]^{\prime}$, which cancels out the nearest pole at $z_1$. This case is similar to the lens-shaped domain with $\alpha=1/2$. Hence, the errors $E_{n,2}(K,G)$, $E_{n,\infty}(f_0,G)$, $\widetilde{E}_{n,2}(K,G)$ and $\widetilde{E}_{n,\infty}(f_0,G)$ satisfy respectively (\ref{e:1}), (\ref{e:2}), (\ref{e:3}) and (\ref{e:4}).
Our purpose here, is to illustrate that the error bounds in (\ref{e:1})--(\ref{e:4}) reflect the actual errors.
We do so by computing estimates to $\varrho_n$ and $\varrho_n^\star$ of $\varrho$ by using (\ref{eq:En-rho})--(\ref{e:7}).

In Table~\ref{tab:5.5}, the results associated with the errors $E_{n,2}(K,G)$ and $\widetilde{E}_{n,2}(K,G)$ indicate clearly the convergence of $\varrho_{n}$ to $\varrho$. Regarding the errors $E_{n,\infty}(f_0,G)$ and $\widetilde{E}_{n,\infty}(f_{0},G)$, the results of Table~\ref{tab:5.6} show that $\varrho_{n}^{\star}$
converges faster to $\varrho$ than $\varrho_{n}$. This suggest a behavior of the type (\ref{e:8}) for $E_{n,\infty}(f_0,G)$ and $\widetilde{E}_{n,\infty}(f_{0},G)$. In both tables the numbers confirm the remarkable
advantage of the BKM/AB over the BKM.

\begin{table}[t]
\begin{center}
\begin{tabular}{|c|c|c|c|c|c|c|c|c|}
\hline
${\ } $
& \multicolumn{2}{|c}{\textrm{BKM:}$\,\, \varrho\approx1.452$}
& \multicolumn{2}{c|}{\textrm{BKM/AB:}$\,\, \varrho\approx2.212$}
${\vphantom{{\sum^{\sum^N}}}}$ \\*[3pt]
\cline{2-5}
$n$ & $E_{n,2}(K,G)$ & $\varrho_{n}$ &
$\widetilde{E}_{n,2}(K,G)$ &
$\varrho_{n}$
${\vphantom{\sum^{\sum^{\sum^N}}}}$ \\*[3pt]  \hline
5  & 7.1e-02 &    -   &2.8e-02&   -  \\
10 & 1.6e-02 & 1.55 &7.2e-04&2.39\\
15 & 2.8e-03 & 1.54 &1.7e-05&2.29\\
20 & 5.1e-04 & 1.49 &3.6e-07&2.29\\
25 & 8.7e-05 & 1.49 &7.6e-09&2.26\\
30 & 1.5e-05 & 1.48 &1.6e-10&2.24\\
35 & 2.4e-06 & 1.47 &3.2e-12&2.24\\
40 & 4.0e-07 & 1.47 &6.4e-14&2.24\\
45 & 6.6e-08 & 1.47 &1.3e-15&2.23\\
50 & 1.1e-08 & 1.46 &2.6e-17&2.23\\ \hline
\end{tabular}
\end{center}

\medskip
\caption{BKM approximations to $K$: Half-disk.}
\label{tab:5.5}
\end{table}

\begin{table}[t]
\begin{center}
\begin{tabular}{|c|c|c|c|c|c|c|c|c|}
\hline
${\ } $
& \multicolumn{3}{|c}{\textrm{BKM:}$\,\, \varrho\approx1.452$}
& \multicolumn{3}{c|}{\textrm{BKM/AB:}$\,\, \varrho\approx2.212$}
${\vphantom{{\sum^{\sum^N}}}}$ \\*[3pt]
\cline{2-7}
   $n$ &$E_{n,\infty}(f_{0},G)$ & $\varrho_{n}^{\star}$
   &$\varrho_{n}$&$\widetilde{E}_{n,\infty}(f_{0},G)$
   & $\varrho_{n}^{\star}$&$\varrho_{n}$
   ${\vphantom{\sum^{\sum^{\sum^N}}}}$ \\*[3pt]  \hline
   5  &1.1e-01&   -  &   -  &4.0e-02&   -  &   -  \\
   10 &2.2e-02&1.401&1.64&7.7e-04&2.20&2.62\\
   15 &3.4e-03&1.446&1.60&1.5e-05&2.20&2.42\\
   20 &5.3e-04&1.450&1.55&2.8e-07&2.22&2.37\\
   25 &8.3e-05&1.452&1.53&5.2e-09&2.22&2.34\\
   30 &1.3e-05&1.452&1.51&9.8e-11&2.21&2.31\\
   35 &2.0e-06&1.452&1.51&1.8e-12&2.22&2.30\\
   40 &3.1e-07&1.452&1.50&3.5e-14&2.20&2.27\\
   45 &4.8e-08&1.452&1.49&6.6e-16&2.21&2.27\\
   50 &7.4e-09&1.452&1.49&1.2e-17&2.22&2.27\\ \hline
\end{tabular}
\end{center}

\medskip
\caption{BKM approximations to $f_{0}$:  Half-disk.}
\label{tab:5.6}
\end{table}

\medskip
\noindent
\textit{Case (ii)}: In this case $f_0$ has a branch point singularity at the point $\tau_1=0$ with
$$
f_0(z)=f_0(0)+\sum_{j=1}^{\infty}a_jz^{j/\alpha},\,\,\,a_1\neq0,
$$
valid for $z$ close to $0$. The nearest singularity of $f_{0}$ in $\Omega$ is a simple pole at $z_{1}=4$, where $|\Phi(z_{1})|\approx 2.04$. For the application of BKM, Theorem~\ref{th3} gives that
\begin{equation}\label{eq:47}
E_{n,2}(K,G) \leqslant c_{1}\frac{1}{n^{1/3}}+c_{2}\frac{1}{R^{n}},
\end{equation}
and
\begin{equation}\label{eq:48}
E_{n,\infty}(f_0,G)\leqslant c_{3}\sqrt{\log n}\frac{1}{n^{1/3}}+c_{4}\frac{1}{R^{n}},
\end{equation}
where $1<R<|\Phi(z_{1})|$.

Since $1/|\Phi(z_1)|^{50}\approx 3.3\times 10^{-16}$, and in view of Theorem~\ref{th:c}, we include in our basis only singular functions that reflect the branch point singularity of $f_0$ at $\tau_1$.
More precisely, in order to keep the contribution of both sources of error balanced, we choose to use the first $15$ singular function of the form $z^{j/\alpha-1}$, where $j/\alpha\notin\mathbb{N}$. This gives
$s^\star=23/3$ in Theorem~\ref{th:c}, and hence the following estimates for the errors in the resulting BKM/AB approximations,
\begin{equation}\label{eq:49}
\widetilde{E}_{n,2}(K,G)\leq c_{1}\frac{1}{n^{23/3}}+ c_{2}\frac{1}{R^{n}},
\end{equation}
and
\begin{equation}\label{eq:50}
\widetilde{E}_{n,\infty}(f_0,G)\leqslant
c_{3}\sqrt{\log n}\frac{1}{n^{23/3}}+c_{4}\frac{1}{R^{n}},
\end{equation}
where $ 1<R<|\Phi(z_{1})|$.

Below, we present numerical results that illustrate the rates in (\ref{eq:49})--(\ref{eq:50}),
where we use the following notation:
\begin{itemize}
\item
$\sigma:$ This denotes the exponent of $1/n$ in the errors $(\ref{eq:49})$--$(\ref{eq:50})$.
\item
$\sigma_n:$ This denotes the estimate of $\sigma$ corresponding to $n$, and is determined as follows: With $E_{n}$  denoting any of the two errors $E_{n,2}(K,G)$, $\widetilde{E}_{n,2}(K,G)$, we assume that
\begin{equation}
E_{n}\approx c\frac{1}{n^{\sigma}}
 \end{equation}
 and seek to estimate $\sigma$ by means of the formula
\begin{equation}
\sigma_{n}=\log\bigg(\frac{E_{n-5}}{E_{n}}\bigg)/\log\bigg(\frac{n}{n-5}\bigg).
\end{equation}
If $E_{n}$  denotes either of the two errors $E_{n,\infty}(f_0,G)$ or $\widetilde{E}_{n,\infty}(f_{0},G)$,
then we assume that
\begin{equation}
E_{n}\approx c\sqrt{\log n}\frac{1}{n^{\sigma}},
\end{equation}
and seek to estimate $\sigma$ by means of the formula
\begin{equation}
\sigma_{n}=\frac{\log\big(\frac{E_{n-5}}{E_{n}}\big)-\frac{1}{2}\log\big[\frac{\log (n-5)}{\log n}\big]}{\log\big(\frac{n}{n-5}\big)}.
\end{equation}
\end{itemize}
In addition, we check a behavior of the form (\ref{e:9}) for the errors $\widetilde{E}_{n,2}(K,G)$ and $\widetilde{E}_{n,\infty}(f_0,G)$, by computing $\varrho_n$ as in (\ref{e:10}), with $5$ in the place of $4$.

Our purposes here is to show that the change of the dominant term in both (\ref{eq:49}) and (\ref{eq:50})
can actually be detected in the computed errors. This is indeed the case in the results quoted in
Table~\ref{tab:5.15}. More precisely, the results associated with the errors  $\widetilde{E}_{n,2}(K,G)$ and $\widetilde{E}_{n,\infty}(f_{0},G)$ indicate the convergence of $\varrho_{n}$ to $\varrho$ for values of $n$ up to $50$ and the convergence of $\sigma_{n}$ to $\sigma$ for values larger than $50$.
Furthermore, the results show  that the two constants $c_1$ and $c_2$ in (\ref{eq:49}) and  $c_3$ and $c_4$ in (\ref{eq:50}) are, respectively,
of the same magnitude.

\begin{table}[t]
 \begin{center}
 \begin{tabular}{|c|c|c|c|c|c|c|c|c|c|c|c|c|c|}
 \hline
 ${\ } $
& \multicolumn{6}{|c|}{BKM/AB: $\quad\sigma \approx 7.67$ $\quad \varrho\approx 2.04$}
${\vphantom{{\sum^{\sum^N}}}}$ \\*[3pt]
\cline{2-7}
   $n$ & $\widetilde{E}_{n,2}(K,G)$ & $\sigma_n$ &$\varrho_n$&
    $\widetilde{E}_{n,\infty}(f_{0},G)$ &
    $\sigma_n$ &$\varrho_n$
     ${\vphantom{\sum^{\sum^{\sum^N}}}}$ \\*[3pt]  \hline
   20 & 7.2e-05 &   -    &  -   &8.2e-05&  -    &  -  \\
   25 & 1.6e-05 & 6.74   & 1.35 &1.5e-05& 7.62  & 1.40\\
   30 & 2.9e-06 & 9.31   & 1.41 &2.6e-06& 9.70  & 1.42\\
   35 & 2.2e-07 & 16.84  & 1.67 &1.8e-07& 17.37 & 1.71\\
   40 & 1.0e-08 & 22.81  & 1.84 &7.7e-09& 23.44 & 1.86\\
   45 & 4.1e-10 & 27.41  & 1.90 &2.8e-10& 28.17 & 1.94\\
   50 & 1.3e-11 & 32.83  & 1.99 &1.0e-11& 31.25 & 1.95\\
   55 & 7.5e-12 & 5.84   & 1.12 &5.3e-12& 7.05  & 1.14\\
   60 & 2.6e-12 & 11.84  & 1.23 &2.0e-12& 11.20 & 1.22\\
   65 & 1.3e-12 & 8.68   & 1.15 &9.9e-13& 8.73  & 1.15\\
   70 & 7.4e-13 & 7.50   & 1.12 &5.9e-13& 7.04  & 1.11\\
   75 & 4.4e-13 & 7.58   & 1.11 &3.5e-13& 7.57  & 1.11\\
   80 & 2.7e-13 & 7.66   & 1.10 &2.1e-13& 7.62  & 1.10\\ \hline
\end{tabular}
\end{center}

\medskip
\caption{BKM approximations to $f_{0}$ and $K$: $3/4$-disk.}
\label{tab:5.15}
\end{table}

\subsection{Rates of decrease of the Bergman polynomials.}\label{sec:5.3}
First, we present results illustrating the rate of decrease of the sequence $\{P_n(1)\}$ for the circular sector considered in Section \ref{subsec:5.2}, with $\alpha=2/5$.
In this case, the nearest singularities of $f_{0}$ in $\Omega$ are the two simple poles at the symmetric points $z_{1}=e^{2i\pi/5}$, $z_{2}=e^{-2i\pi/5}$, where $|\Phi(z_{1})|=|\Phi(z_{2})|\approx 1.145$.
From the proof of Corollary~\ref{cor:thm3.1} and (\ref{eq:47}) we have that
\begin{equation}\label{e:15}
|P_n(z_0)|\leq\|K(\cdot,z_0)-K_n(\cdot,z_0)\|_{L^2(G)}\leq c_1\frac{1}{n^{(2-\alpha)/\alpha}}+c_2\frac{1}{R^n},
\end{equation}
where $1<R<|\Phi(z_{1})|$,
Accordingly, we check to detect the decay in the following two forms:
\begin{equation}\label{eq:Pndecay}
|P_n(1)|\approx c\frac{1}{\varrho^{n}}
\quad\mbox{and}\quad
|P_n(1)|\approx c\frac{1}{n^{\sigma}},
\end{equation}
with $\varrho=|\Phi(z_1)|$ and, in view of the remark made in \cite[pp.~530--531]{MSS}, $\sigma=(2-\alpha)/\alpha+1/2$. We do so, by estimating $\varrho$ and $\sigma$, respectively, by means of the formulas
\begin{equation}
\varrho_{n}=\left(\frac{|P_{n-10}(1)|}{|P_{n}(1)|}\right)^{\frac{1}{10}},
\end{equation}
and
\begin{equation}
\sigma_{n}=\log\bigg(\frac{|P_{n-10}(1)|}{|P_{n}(1)|}\bigg)/\log\bigg(\frac{n}{n-10}\bigg).
\end{equation}

The results listed in Table~\ref{tab:5.9} show clearly the transition from one dominant term to the other in (\ref{e:15}) for values of $n$ around 50.

\begin{table}[t]
\begin{center}
\begin{tabular}{|c|c|c|c|c|}
\hline
${\ } $
& \multicolumn{2}{|c}{$\,\, \sigma=4.5$}
& \multicolumn{1}{c|}{$\,\, \varrho\approx 1.145$}
${\vphantom{{\sum^{\sum^N}}}}$ \\*[3pt]
\cline{2-4}
   $n$ & $|P_{n}(1)|$ & $\sigma_{n}$ & $\varrho_n$
    ${\vphantom{\sum^{\sum^{\sum^N}}}}$ \\*[3pt]  \hline
   10  & 2.6e-02 &  - &   -   \\
   20  & 1.2e-03 &4.51& 1.37  \\
   30  & 7.6e-06 &12.38& 1.65 \\
   40  & 1.7e-06 &5.14&  1.16 \\
   50  & 4.0e-07 &6.57&  1.16 \\
   60  & 1.8e-07 &4.35&  1.08 \\
   70  & 9.1e-08 &4.49&  1.07 \\
   80  & 5.0e-08 &4.50&  1.06 \\
   90  & 2.9e-08 &4.50&  1.05 \\
   100 & 1.8e-08 &4.50&  1.05 \\ \hline
\end{tabular}
\end{center}

\medskip
\caption{Rate of decrease of $|P_{n}(1)|$: Circular sector, $\alpha=2/5$.}
\label{tab:5.9}
\end{table}

We end, by presenting results that illustrate the rate of decrease of the augmented sequence $\{\widetilde{P}_n(1)\}$, for the circular sector considered in Section \ref{subsec:5.2},
where now we consider the two cases $\alpha=3/4$ and $\alpha=4/5$.
When $\alpha=3/4$, then the nearest singularities of $f_{0}$ in $\Omega$ are the two simple poles at the symmetric points $z_{1}=e^{3i\pi/4}$ and $z_{2}=e^{-3i\pi/4}$, where $|\Phi(z_{1})|=|\Phi(z_{2})|\approx1.349$. When $\alpha=4/5$, then the nearest singularities of $f_{0}$ in $\Omega$ are the two simple poles at the symmetric points $z_{1}=e^{4i\pi/5}$ and $z_{2}=e^{-4i\pi/5}$, where $|\Phi(z_{1})|=|\Phi(z_{2})|\approx1.372$.
In both cases, we construct the sequence $\{\widetilde{P}_n(z)\}$ by augmenting the monomial basis functions
with the singular function $z^{1/\alpha-1}$, which reflects the branch point singularity of $f_0$ at $\tau_1=0$,
and we seek to detect the decay of the sequence $\{\widetilde{P}_n(1)\}$ in the form
$$
 |\widetilde{P}_n(1)|\approx c\frac{1}{n^{\sigma}},
$$
where, in view of Theorem~\ref{th:c} and \cite[pp.~530--531]{MSS}, $\sigma=2(2-\alpha)/\alpha+1/2$.
As above, we estimate $\sigma$ by means of the formula
\begin{equation}
\sigma_{n}=\log\bigg(\frac{|\widetilde{P}_{n-10}(1)|}{|\widetilde{P}_{n}(1)|}\bigg)/\log\bigg(\frac{n}{n-10}\bigg).
\end{equation}

The results listed in Tables~\ref{tab:5.11} and \ref{tab:5.13} indicate clearly the convergence of $\sigma_n$ to the predicted value of $\sigma$, indicating that the argument in \cite[pp.~530--531]{MSS} applies also to the case
of the augmented Bergman polynomials.

\begin{table}[t]
\begin{center}
\begin{tabular}{|c|c|c|c|}
   \hline
   ${\ } $
   & \multicolumn{1}{|l}{$\quad \sigma\approx3.833$}&
${\vphantom{{\sum^{\sum^N}}}}$ \\*[3pt]
\cline{1-3}
   $n$ & $|\widetilde{P}_{n}(1)|$ & $\sigma_{n}$
   ${\vphantom{\sum^{\sum^{\sum^N}}}}$ \\*[3pt]  \hline
   10  & 2.8e-03 &  -  \\
   20  & 7.2e-05 &5.30 \\
   30  & 1.1e-05 &4.60 \\
   40  & 3.2e-06 &4.36 \\
   50  & 1.3e-06 &3.86 \\
   60  & 6.6e-07 &3.89 \\
   70  & 3.6e-07 &3.89 \\
   80  & 2.2e-07 &3.88 \\
   90  & 1.4e-07 &3.88 \\
   100 & 9.1e-08 &3.87 \\ \hline
\end{tabular}
\end{center}

\medskip
\caption{Rate of decrease of $|\widetilde{P}_{n}(1)|$: Circular sector, $\alpha=3/4$.}
\label{tab:5.11}
\end{table}

\begin{table}[t]
\begin{center}
\begin{tabular}{|c|c|c|c|}
   \hline
   ${\ } $
   & \multicolumn{1}{|l}{$\quad \sigma=3.5$}&
${\vphantom{{\sum^{\sum^N}}}}$ \\*[3pt]
\cline{1-3}
   $n$ & $|\widetilde{P}_{n}(1)|$ & $\sigma_{n}$
   ${\vphantom{\sum^{\sum^{\sum^N}}}}$ \\*[3pt]  \hline
   10  & 2.3e-03 &  -  \\
   20  & 2.7e-04 &3.10 \\
   30  & 2.2e-05 &6.17 \\
   40  & 8.7e-06 &3.20 \\
   50  & 3.8e-06 &3.72 \\
   60  & 2.0e-06 &3.66 \\
   70  & 1.1e-06 &3.63 \\
   80  & 6.9e-07 &3.61 \\
   90  & 4.5e-07 &3.59 \\
   100 & 3.1e-07 &3.58 \\ \hline
\end{tabular}
\end{center}

\medskip
\caption{Rate of decrease of
$|\widetilde{P}_{n}(1)|$: Circular sector, $\alpha=4/5$.}
\label{tab:5.13}
\end{table}

%\clearpage

\bibliographystyle{amsplain}

\def\cprime{$'$}
\providecommand{\bysame}{\leavevmode\hbox to3em{\hrulefill}\thinspace}
\providecommand{\MR}{\relax\ifhmode\unskip\space\fi MR }
\providecommand{\MRhref}[2]{%
  \href{http://www.ams.org/mathscinet-getitem?mr=#1}{#2}
}
\providecommand{\href}[2]{#2}

\end{document}